\documentclass[reqno,11pt]{amsart}

\usepackage{latexsym}
\usepackage{float}
\usepackage{amssymb}
\usepackage{amsmath}
\usepackage{amsthm}
\usepackage{mathrsfs}
\usepackage{euscript}
\usepackage[cmtip,all]{xy}

\usepackage{bbold}
\usepackage{tabu}
\usepackage{enumerate}

\usepackage{euscript}

\usepackage{graphicx}

\usepackage{bm}
\usepackage{stackrel}

\newcommand{\Mod}[1]{\ (\mathrm{mod}\ #1)}

\makeatletter
\g@addto@macro \normalsize {%
 \setlength\abovedisplayskip{7pt}%
 \setlength\belowdisplayskip{6pt}%
}
\makeatother

\newtheorem{thm}[equation]{Theorem}

\newtheorem{cor}[equation]{Corollary}
\newtheorem{prop}[equation]{Proposition}

\theoremstyle{remark}
\newtheorem{rem}[equation]{Remark}

\theoremstyle{definition}
\newtheorem{defn}[equation]{Definition}

\newtheorem{prob}{Problem}

\numberwithin{equation}{section}

\newcommand{\gb}{\beta}
\newcommand{\ga}{\alpha}

\newcommand{\gL}{\Lambda}

\newcommand{\gD}{\Delta}

\newcommand{\eps}{\varepsilon}

\newcommand{\fa}{{\mathfrak a}}             
\newcommand{\fb}{{\mathfrak b}}
\newcommand{\fg}{{\mathfrak g}}
\newcommand{\fh}{{\mathfrak h}}

\newcommand{\fl}{{\mathfrak l}}
\newcommand{\fm}{{\mathfrak m}}
\newcommand{\fn}{{\mathfrak n}}

\newcommand{\fp}{{\mathfrak p}}

\newcommand{\fy}{{\mathfrak y}}

\newcommand{\f}{\mathfrak}

\newcommand{\R}{\mathbb{R}}          
\newcommand{\C}{\mathbb{C}}          
           
\newcommand{\Z}{\mathbb{Z}}

\newcommand{\ad}{\mathrm{ad}}
\newcommand{\Ad}{\mathrm{Ad}}
\newcommand{\Cal}{\mathcal}

\newcommand{\Hom}{\operatorname{Hom}}

\renewcommand{\Im}{\mathrm{Im}}
\newcommand{\Ind}{\mathrm{Ind}}
\newcommand{\IP}[2]{\langle#1 , #2\rangle}     

\newcommand{\spn}{\text{span}}

\newcommand{\Tr}{\text{Tr}}

\newcommand{\Sol}{\mathrm{Sol}}

\newcommand{\Pol}{\mathrm{Pol}}
\newcommand{\poly}{\mathrm{poly}}

\newcommand{\To}{\longrightarrow}

\newcommand{\Diff}{\mathrm{Diff}}
\newcommand{\Irr}{\mathrm{Irr}}

\newcommand{\triv}{\mathrm{triv}}
\newcommand{\fin}{\mathrm{fin}}

\newcommand{\sym}{\mathrm{sym}}

\newcommand{\D}{\Cal{D}}

\newcommand{\wH}{\widetilde{H}}

\newcommand{\diag}{\mathrm{diag}}

\newcommand{\dpi}{d\pi}

\newcommand{\id}{\mathrm{id}}
\newcommand{\pr}{\mathrm{pr}}

\newcommand{\symb}{\mathrm{symb}}

\newcommand{\sgn}{\mathrm{sgn}}

\newcommand{\std}{\mathrm{std}}

\newcommand{\RP}{\mathbb{R}\mathbb{P}}

\newcommand{\Ker}{\mathrm{Ker}}

\newcommand{\abs}[1]{\left\vert#1\right\vert}

\providecommand*{\donothing}[1]{}

\makeatletter
\@namedef{subjclassname@2020}{%
  \textup{2020} Mathematics Subject Classification}
\makeatother

\begin{document}

\baselineskip=16pt
\tabulinesep=1.2mm


\title[]{
On the intertwining differential operators from a line bundle to a vector bundle 
over the real projective space
 }

\dedicatory{Dedicated to the memory of Gerrit van Dijk}

\author{Toshihisa Kubo}
\author{Bent {\O}rsted}

\address{Faculty of Economics, 
Ryukoku University,
67 Tsukamoto-cho, Fukakusa, Fushimi-ku, Kyoto 612-8577, Japan}
\email{toskubo@econ.ryukoku.ac.jp}

\address{Department of Mathematics, 
Aarhus University,
Ny Munkegade 118 DK-8000 Aarhus C, Denmark}
\email{orsted@imf.au.dk}


\date{\today}

\maketitle


\begin{abstract} 
We classify and construct $SL(n,\R)$-intertwining differential operators $\D$ 
from a  line bundle to a vector bundle 
over the real projective space $\RP^{n-1}$ by the F-method. 
This generalizes a classical result of Bol for $SL(2,\R)$.
Further, we classify the $K$-type formulas for 
the kernel $\Ker(\D)$ and image $\Im(\D)$ of $\D$.
The standardness of the homomorphisms $\varphi$ corresponding to
the differential operators $\D$ between
generalized Verma modules are also discussed.
\end{abstract}


\section{Introduction}\label{sec:intro}

Gerrit van Dijk worked on harmonic analysis on $p$-adic groups and real Lie groups, inspired by Harish-Chandra as a post-doc at IAS Princeton. He studied both abstract questions such as the nature of convolution algebras of functions on symmetric spaces, but also concrete special functions and distributions on such spaces. He also considered induced representations from a parabolic subgroup $P$ to a reductive Lie group $G$ and the explicit structure of such in concrete cases. In particular, he jointly with Molchanov studied the case of $G/P$ being 
the real projective space (\cite{vDM99}).
In this paper we aim to study intertwining differential operators between
such induced representations.

\subsection{Main problems}
 
To state the main problems of this paper, we first introduce some notation.
Let $G$ be a real reductive Lie group with complexified Lie algebra $\fg$. 
Let $P$ be a parabolic subgroup with 
Langlands decomposition $P=MAN_+$. 
We write $\Irr(M)_\fin$ for the set of  equivalence classes of 
finite-dimensional irreducible representations of $M$.
Likewise, let $\Irr(A)$ denote the set of characters of $A$. 
Then, for the outer tensor product $\varpi\boxtimes \nu \boxtimes \triv$ 
of $\varpi \in \Irr(M)_\fin$, $\nu \in \Irr(A)$, and the trivial representation
$\triv$ of $N_+$, we put
\begin{equation*}
I(\varpi,\nu) = \Ind_{P}^G(\varpi\boxtimes \nu \boxtimes \triv)
\end{equation*}
for an (unnormalized) parabolically induced representation of $G$.
Let $\Diff_G(I(\sigma,\lambda), I(\varpi,\nu))$ denote the space of $G$-intertwining
differential operators $\D \colon I(\sigma,\lambda) \to I(\varpi,\nu)$.

There are two main problems in this paper. The first problem concerns
the classification and construction of 
intertwining differential operators $\D \in \Diff_G(I(\sigma,\lambda), I(\varpi,\nu))$
as follows.
 
\begin{prob}\label{prob:A}
Do the following.
\begin{enumerate}
\item[(A1)] Classify $(\sigma, \lambda), (\varpi, \nu) \in \Irr(M)_\fin\times \Irr(A)$ 
such that
\begin{equation*}
\Diff_G(I(\sigma,\lambda), I(\varpi,\nu))\neq\{0\}.
\end{equation*}

\item[(A2)] 
For $(\sigma, \lambda), (\varpi, \nu) \in \Irr(M)_\fin\times \Irr(A)$ classified in (A1),
determine the dimension 
\begin{equation*}
\dim \Diff_G(I(\sigma,\lambda), I(\varpi,\nu)).
\end{equation*}

\item[(A3)] 
For $(\sigma, \lambda), (\varpi, \nu) \in \Irr(M)_\fin\times \Irr(A)$ classified in (A1),
construct generators 
\begin{equation*}
\D \in \Diff_G(I(\sigma,\lambda), I(\varpi,\nu)).
\end{equation*}
\end{enumerate}
\end{prob}

Let $K$ be a maximal compact subgroup of $G$.
For $\D \in \Diff_G(I(\sigma,\lambda), I(\varpi,\nu))$, the kernel $\Ker(\D)$ and 
image $\Im(\D)$ are $G$-invariant subspaces of $I(\sigma,\lambda)$ and 
$I(\varpi, \nu)$, respectively. Let $\Ker(\D)_K$ and $\Im(\D)_K$ denote the space of 
$K$-finite vectors of $\Ker(\D)$ and $\Im(\D)$, respectively.
The following is the other main problem of this paper.

\begin{prob}\label{prob:B}
Classify the $K$-type formulas of $\Ker(\D)_K$ and $\Im(\D)_K$.
\end{prob}

For instance, if $G= SL(3,\R)$, then the maximal compact subgroup $K$ is
$K=SO(3)$. Problem \ref{prob:B} then 
asks how $\Ker(\D)_K$ and $\Im(\D)_K$ decompose 
as $SO(3)$-modules. 
We shall answer this question in Corollary \ref{cor:IDO4} for $SL(n,\R)$
with $n\geq 3$.

\vskip 0.1in

Problem \ref{prob:A} for the first order intertwining differential operators
for general $(G, P)$ was done around 2000 
independently by  
{\O}rsted \cite{Orsted00},
Slov{\'a}k--Sou{\v c}ek \cite{SS00},
and 
Johnson--Kor{\' a}nyi--Reimann \cite{JKR03}, which
generalize the work of Fegan \cite{Fegan76} for $G = SO(n,1)$.
See also the works of
Kor{\' a}nyi--Reimann \cite{KR00} and
Xiao \cite{Xiao15} as relevant works for this case.
The higher order case is still a work in progress.
For some recent studies, see, for instance, 
Barchini--Kable--Zierau \cite{BKZ08, BKZ09}, 
Kable  \cite{Kable12A, Kable12C, Kable13, Kable15, Kable18a, Kable18b}, 
Kobayashi--{\O}rsted--Somberg--Sou{\v c}ek \cite{KOSS15}, and
the first author \cite{Kubo14a, Kubo14b}.

Problem \ref{prob:A} has been paid attention especially in conformal geometry. 
The Yamabe operator, also known as the conformal Laplacian, is a classical example
of a conformally covariant differential operator (cf.\ \cite{KO03a, LP87}).
In this paper we consider the projective structure in parabolic geometry.
In other words, our aim is to classify and construct ``projectively covariant'' differential operators.

On the study of intertwining differential operators,
the BGG sequence is an important background (cf.\ \cite{CD01, CSS01}).
The kernel  of the first BGG operators has also been studied carefully
from a geometric point of view (cf.\ \cite{CGH12b, CGH12a, GZ23}).
Intertwining differential operators are also studied intensively by Dobrev
in quantum physics
(cf.\ \cite{Dobrev88, Dobrev13, Dobrev16, Dobrev17, Dobrev18, Dobrev19}).

\vskip 0.1in

 We next describe the concrete setup of this paper.

\subsection{Specialization to $(SL(n,\R), P_{1,n-1};(\pm, \triv))$}

In this paper we consider Problems \ref{prob:A} and \ref{prob:B} for 
$(G, P; \sigma) = (SL(n,\R), P_{1,n-1};(\pm, \triv))$, 
where $P_{1,n-1}$ is the maximal parabolic subgroup of $G$ corresponding to
the partition $n=1+(n-1)$ so that $G/P_{1,n-1}$ is diffeomorphic to
the real projective space $\RP^{n-1}$.
The representation $(\pm,\triv) \in \Irr(M)_{\fin}$ denotes a one-dimensional
representation of $M\simeq SL^\pm(n-1,\R)$. The details will be discussed 
in Section \ref{sec:notation}.
We shall solve Problem \ref{prob:A} in Theorems \ref{thm:IDO1} and \ref{thm:IDO},
and Problem \ref{prob:B} in Corollary \ref{cor:IDO4}. 

In 1949, Bol from projective differential geometry showed that 
$SL(2,\R)$-intertwining differential operators for an equivariant line bundle
over $S^1$ 
are only
the powers of the standard derivative $\frac{d^k}{dx^k}$
(cf.\ \cite{Bol49} and \cite[Thm. 2.1.2]{OT05}). 
The operator $\frac{d^k}{dx^k}$ is sometimes called
the \emph{Bol operator}. 
For recent results on analogues of the Bol operator for Lie superalgebras,
see, for instance, \cite{BLS22a, BLS22b}. 
We successfully generalize the classical result
of Bol on $S^1$, which is a double cover of $\RP^1$, to $SL(n,\R)$ on $\RP^{n-1}$.

It is classically known that the Bol operators $\frac{d^k}{dx^k}$ are the residue
operators of the Knapp--Stein operator.
In contrast to the case of $n=2$, the Knapp--Stein operator does not exist
for $(SL(n,\R), P_{1,n-1})$ for $n\geq 3$ (cf.\ \cite{MS14}).
Thus, the differential operators obtained in Theorem \ref{thm:IDO} are 
not the residue operators of such.
In the sense of Kobayashi--Speh \cite{KS18},
our differential operators are all \emph{sporadic operators} for $n \geq 3$.

\subsection{The F-method}
Our main tool to work on Problem \ref{prob:A} is the so-called F-method.
This is a fascinating method invented by Toshiyuki Kobayashi
around 2010 in the course of the study of his branching program
(see, for instance, \cite{Kobayashi15}, \cite[Introduction]{KOSS15},  and \cite[Sect.\ B]{Tanimura19}).
Since then, Problem A has been intensively studied by the F-method
especially in symmetry breaking setting
(cf.\ \cite{FJS20, Kobayashi13, Kobayashi14, KKP16, KP1, KP2, KOSS15, 
KrSom17c, KrSom17a, KrSom17b, MS13, Perez23, Somberg13+}). 

The F-method makes it possible to
classify and construct intertwining differential operators
(or more generally speaking, 
\emph{differential symmetry breaking operators}) by solving
a certain system of partial differential equations.
To describe the main idea more precisely, recall that
it follows from a fundamental work of Kostant \cite{Kostant75} 
that the space $\Diff_G(I(\sigma,\lambda), I(\varpi,\nu))$
of intertwining differential operators
is isomorphic to the space of $(\fg, P)$-homomorphisms between generalized Verma modules (see also \cite{CS90, Dobrev88, HJ82, Huang93, KR00}). 
Schematically, we have
\begin{equation}\label{eqn:SHD}
\textnormal{Hom}(\text{Verma}) 
\stackrel{\sim}{\To}
\Diff_G(I(\sigma,\lambda), I(\varpi,\nu)).
\end{equation}
The precise statement will be given in Theorem \ref{thm:duality}.
In general, it is easier to work with the Verma module side 
``$\Hom(\text{Verma})$''
than the differential operator side $\Diff_G(I(\sigma,\lambda), I(\varpi,\nu))$.
Thus, the standard strategy to tackle Problem \ref{prob:A} is to convert the problem
into the one for generalized Verma modules. Nonetheless,
even in a case that the unipotent radical $N_+$ is abelian, 
it requires involved combinatorial computations 
(see, for instance, \cite[Chap.\ 5]{Juhl09}).

The novel idea of the F-method is to further identify
$\Diff_G(I(\sigma,\lambda), I(\varpi,\nu))$ 
with the space 
of polynomial solutions to a system of 
PDEs by applying a Fourier transform
to a generalized Verma module. In other words, the F-method puts another picture
``$\textnormal{Sol}(\text{PDE})$'' to \eqref{eqn:SHD} as follows.
\begin{equation}\label{eqn:SHD2}
\xymatrix@R-=1.0pc@C+=0.3cm{
& \textnormal{Sol}(\text{PDE}) 
\ar@{.>}[rdd]_{\sim}
 &\\
&                                    & \\
\textnormal{Hom}(\text{Verma}) 
 \ar@{->}[rr]^{\hspace{10pt}\sim}
\ar@{->}[ruu]_{\sim}
 & & 
\Diff_G(I(\sigma,\lambda), I(\varpi,\nu)).
}
\end{equation}
Then, in the F-method, one achieves the classification and construction of
intertwining differential operators simultaneously by solving the system of PDEs. 
We shall explain more details in Section \ref{sec:Fmethod}.

\vskip 0.1in

\subsection{The counterpart of generalized Verma modules}

As observed above, one can obtain $(\fg,P)$-homomorphisms in 
``$\textnormal{Hom}(\text{Verma})$'' from the solutions in 
``$\textnormal{Sol}(\text{PDE})$''. We thus study the algebraic counterpart
of Problem \ref{prob:A} for $(\fg,P)=(\f{sl}(n,\C), P_{1,n-1})$. 
It is achieved in Theorems \ref{thm:Hom1} and \ref{thm:Hom}.
We further consider a variant of Problem \ref{prob:A} for $\fg$-homomorphisms
between generalized Verma modules in Theorem \ref{thm:Hom2}.
We also classify in Corollary \ref{cor:Red} the reducibility 
of the generalized Verma module in consideration; this recovers
 a result of \cite{BX21, He15, HKZ19} for the pair $(\fg,\fp)=(\f{sl}(n,\C),\fp_{1,n-1})$.

A $\fg$-homomorphism between generalized Verma modules is called 
a \emph{standard map} if it is induced from a $\fg$-homomorphism between
the corresponding (full) Verma modules;
otherwise, it is called a \emph{non-standard map} (\cite{Boe85, Lepowsky77}).
In this paper we also show that the resulting $\fg$-homomorphisms in 
Theorem \ref{thm:Hom2} are all standard. 

\subsection{The $K$-type formulas for $\Ker(\D)_K$ and $\Im(\D)_K$}
\label{sec:Ktype}

Kable \cite{Kable11} and the authors \cite{KuOr19} recently showed 
a Peter--Weyl type theorem for the kernel $\Ker(\D)$ of
an intertwining differential operator $\D$.
The theorem allows us to compute  
the $K$-type formula of $\Ker(\D)$ explicitly by solving the hypergeometric/Heun differential equation (\cite{KuOr21+, KuOr19}).
Tamori \cite{Tamori21} independently used a similar idea 
to determine the $K$-type formula of $\Ker(\D)$ for his study of 
minimal representations.

The Peter--Weyl type theorem works nicely for 
first and second order differential operators; nonetheless,
it requires a certain amount of computations for higher order cases.
Since the differential operators that we obtained in Theorem \ref{thm:IDO} 
have arbitrary order, we take another approach in this paper.

In 1990, van Dijk--Molchanov \cite{vDM99} and Howe--Lee \cite{HL99}
independently showed among other things that the degenerate principal series representations in consideration have length two 
and have a unique finite-dimensional irreducible subrepresentation $F$.
(See M{\"o}llers--Schwarz \cite{MS14} for recent development of this matter.)
Since the $K$-type structures of the induced representations are also known, 
the determination of the $K$-type formula of $\Ker(\D)$ 
is equivalent to showing $\Ker(\D) \neq \{0\}$. As the length is two,
it simultaneously determines the $K$-type formula of $\Im(\D)$.

In order to show that $\Ker(\D) \neq \{0\}$, we show $\D f_{0}=0$ 
for a lowest weight vector $f_0 \in F$.
We shall discuss the details in Section \ref{sec:Sol}.
We remark that one can also show $\Ker(\D) \neq \{0\}$ by using
the lowest $K$-type in \cite{Kable12C}  for $n=3$.

\subsection{Organization of the paper}

Now we outline the rest of this paper.
There are seven sections including the introduction.
In Section \ref{sec:Fmethod}, we review a general idea of the F-method.
In particular, a recipe of the F-method will be given in Section \ref{sec:recipe}.
Since we do not find a thorough exposition 
of the F-method in the English literature 
(except the original papers of Kobayashi with his collaborators 
\cite{KP1, KP2, KOSS15}),
we decided to give some detailed account.
We hope that it will be helpful for a wide range of readers.
It is remarked that part of the section is an English translation of the Japanese
article \cite{Kubo22} of the first author.

Section \ref{sec:SLn} is for the specialization of the framework discussed in
Section \ref{sec:Fmethod} to the case 
$(G,P)=(SL(n,\R), P_{1,n-1})$.
In this section we fix some notation and normalizations for the rest of the sections.
Then, in Section \ref{sec:Dphi},  
we give our main results for Problem \ref{prob:A}
for the triple $(G,P;\sigma)=(SL(n,\R), P_{1,n-1};(\pm, \triv))$. These are accomplished
in Theorems \ref{thm:IDO1} and \ref{thm:IDO}. Further, a variant of
Problem \ref{prob:A} for 
$(\fg, P)$-homomorphisms between generalized Verma modules
is also discussed in this section. These are stated in Theorems \ref{thm:Hom1} and \ref{thm:Hom}. 
We give proofs of these theorems in Section \ref{sec:proof} by following the recipe of the F-method.

We consider Problem \ref{prob:B} in Section \ref{sec:Sol}. In this section we classify 
the $K$-type formulas of the kernel $\Ker(\D)$ and the image $\Im(\D)$ of the 
intertwining differential operators $\D$ classified in Section \ref{sec:Dphi}.
These are obtained in Corollary \ref{cor:IDO4}.

Section \ref{sec:Std} is an appendix discussing the standardness of 
the $\fg$-homomorphisms $\varphi$ obtained in Section \ref{sec:Dphi} 
between generalized Verma modules. We first quickly review the definition of the 
standard map. Then, by applying a version of Boe's criterion, we show that  
the homomorphisms $\varphi$ are all standard maps. This is achieved in 
Theorem \ref{thm:std}.

\section{Quick review on the F-method}
\label{sec:Fmethod}

The aim of this section is to review the F-method.
In particular, a recipe of the F-method is given in Section \ref{sec:recipe}.
In Section \ref{sec:proof}, we shall follow the recipe to 
classify and construct intertwining differential operators $\D$.
We mainly follow the arguments in \cite{Kobayashi18} and  \cite{KP1} in this section.

\subsection{General framework}\label{sec:prelim}

Let $G$ be a real reductive Lie group and $P=MAN_+ $ a Langlands decomposition of a parabolic subgroup $P$ of $G$. We denote by $\fg(\R)$ and 
$\fp(\R) = \fm(\R) \oplus \fa(\R) \oplus \fn_+(\R)$ the Lie algebras of $G$ and 
$P=MAN_+$, respectively.

 For a real Lie algebra $\f{y}(\R)$, we write $\f{y}$
and $\Cal{U}(\fy)$ for its complexification and the universal enveloping algebra of 
$\fy$, respectively. 
For instance, $\fg, \fp, \fm, \fa$, and $\fn_+$ are the complexifications of $\fg(\R), \fp(\R), \fm(\R), \fa(\R)$, and $\fn_+(\R)$, 
respectively.

For $\lambda \in \fa^* \simeq \Hom_\R(\fa(\R),\C)$,
we denote by $\C_\lambda$ 
the one-dimensional representation of $A$ defined by 
$a\mapsto a^\lambda:=e^{\lambda(\log a)}$. 
For a finite-dimensional irreducible 
representation $(\sigma, V)$ of $M$ and $\lambda \in \fa^*$,
we denote by $\sigma_\lambda$ the outer tensor representation $\sigma \boxtimes \C_\lambda$. As a representation on $V$, we define $\sigma_\lambda \colon
ma \mapsto a^\lambda\sigma(m)$. By letting $N_+$ act trivially, we regard 
$\sigma_\lambda$ as a representation of $P$. Let $\Cal{V}:=G \times_P V \to G/P$
be the $G$-equivariant vector bundle over the real flag variety $G/P$
 associated with the 
representation $(\sigma_\lambda, V)$ of $P$. We identify the Fr{\' e}chet space 
$C^\infty(G/P, \Cal{V})$ of smooth sections with 
\begin{equation*}
C^\infty(G, V)^P:=\{f \in C^\infty(G,V) : 
f(gp) = \sigma_\lambda^{-1}(p)f(g)
\;\;
\text{for any $p \in P$}\},
\end{equation*}
the space of $P$-invariant smooth functions on $G$.
Then, via the left regular representation $L$ of $G$ on $C^\infty(G)$,
we realize the parabolically induced representation 
$\pi_{(\sigma, \lambda)} = \Ind_{P}^G(\sigma_\lambda)$ on $C^\infty(G/P, \Cal{V})$.
We denote by $R$ the right regular representation of $G$ on $C^\infty(G)$.

Similarly, for a finite-dimensional irreducible
representation $(\eta_\nu, W)$ of $MA$, we define
an induced representation $\pi_{(\eta, \nu)}=\Ind_P^G(\eta_\nu)$ on the 
space $C^\infty(G/P, \Cal{W})$ of smooth sections for a $G$-equivariant vector bundle 
$\Cal{W}:=G\times_PW \to G/P$.
We write $\Diff_G(\Cal{V},\Cal{W})$ for the space of intertwining differential operators
$\D \colon C^\infty(G/P, \Cal{V}) \to C^\infty(G/P, \Cal{W})$.

Let $\mathfrak{g}(\R)=\mathfrak{n}_-(\R) \oplus \mathfrak{m}(\R) 
\oplus \mathfrak{a}(\R) \oplus \mathfrak{n}_+(\R)$ be the 
Gelfand--Naimark decomposition of $\mathfrak{g}(\R)$,
and write $N_- = \exp(\fn_-(\R))$. We identify $N_-$ with the 
open Bruhat cell $N_-P$ of $G/P$ via the embedding 
$\iota\colon N_- \hookrightarrow G/P$, $\bar{n} \mapsto \bar{n}P$.
Via the restriction of the vector bundle $\Cal{V} \to G/P$ to the open Bruhat cell
$N_-\stackrel{\iota}{\hookrightarrow} G/P$,
we regard $C^\infty(G/P,\mathcal{V})$ as a subspace of 
$C^\infty(N_-) \otimes V$.

We view intertwining differential operators 
$\D \colon
C^\infty(G/P,\mathcal{V})
\to C^\infty(G/P,\mathcal{W})$
as differential operators
$\D' \colon C^\infty(N_-) \otimes V
\to C^\infty(N_-) \otimes W$ such that
the restriction $\D'\vert_{C^\infty(G/P,\mathcal{V})}$
to $C^\infty(G/P,\mathcal{V})$ is a map
$\D'\vert_{C^\infty(G/P,\mathcal{V})}\colon
C^\infty(G/P,\mathcal{V})
\to C^\infty(G/P,\mathcal{W})$ (see \eqref{eqn:21} below).
\begin{equation}\label{eqn:21}
\xymatrix{
C^\infty(N_-) \otimes V 
\ar[r]^{\mathcal{D}' } 
& C^\infty(N_-) \otimes W\\ 
C^\infty(G/P,\mathcal{V}) 
\;\; \ar[r]_{\stackrel{\phantom{a}}{\hspace{20pt}\mathcal{D}=\mathcal{D}' \vert_{\small{C^\infty(G/P,\mathcal{V})}}} }
 \ar@{^{(}->}[u]^{\iota^*}
& \;\; C^\infty(G/P,\mathcal{W}) \ar@{^{(}->}[u]_{\iota^*}
}
\end{equation}

\noindent
In particular, we regard $\Diff_G(\Cal{V},\Cal{W})$ as 
\begin{equation}\label{eqn:DN}
\Diff_G(\Cal{V},\Cal{W}) \subset 
\Diff_\C(C^\infty(N_-)\otimes V, C^\infty(N_-)\otimes W).
\end{equation}

To describe the F-method, one needs to introduce the following:

\begin{enumerate}
\item infinitesimal representation $d\pi_{(\sigma,\lambda)}$ (Section \ref{sec:dpi1}),
\vskip 0.05in

\item duality theorem (Section \ref{sec:duality}),
\vskip 0.05in

\item algebraic Fourier transform $\widehat{\;\;\cdot\;\;}$ of Weyl algebras
(Section \ref{sec:Weyl}),
\vskip 0.05in

\item Fourier transformed representation 
$\widehat{d\pi_{(\sigma,\lambda)^*}}$ (Section \ref{sec:dpi2}),
\vskip 0.05in

\item algebraic Fourier transform $F_c$ of generalized Verma modules
(Section \ref{sec:Fc}).

\end{enumerate}


\noindent
After reviewing these objects, we shall discuss the F-method 
in Section \ref{sec:Fmethod2}.

\subsection{Infinitesimal representation $d\pi_{(\sigma,\lambda)}$}
\label{sec:dpi1}

We start with the representation $d\pi_{(\sigma, \lambda)}$ of $\fg$
on $C^\infty(N_-)\otimes V$ derived from the induced representation 
$\pi_{(\sigma, \lambda)}$ of $G$ on $C^\infty(G/P,\Cal{V})$.

For a representation $\eta$ of $G$, we denote by $d\eta$ 
the infinitesimal representation of $\fg(\R)$. 
For instance,  $dL$ and $dR$ denote the infinitesimal representations $\fg(\R)$ of
the left and right regular representations of $G$ on $C^\infty(G)$.
As usual, we naturally extend representations of $\fg(\R)$ 
to ones for its universal enveloping algebra $\Cal{U}(\fg)$ of its complexification $\fg$.
The same convention is applied for closed subgroups of $G$.

For $g \in N_-MAN_+$, we write
\begin{equation*}
g = p_-(g)p_0(g)p_+(g),
\end{equation*}
where $p_\pm(g) \in N_{\pm}$ and $p_0(g) \in MA$. 
Similarly, for $Y \in \fg = \fn_- \oplus \fl\oplus \fn_+$ with $\fl= \fm \oplus \fa$,
we write
\begin{equation*}
Y=Y_{\fn_-} + Y_{\fl} + Y_{\fn_+},
\end{equation*}
where $Y_{\fn_{\pm}} \in \fn_{\pm}$ and $Y_\fl \in \fl$.

Let $X\in \fg(\R)$.
Since $N_-MAN_+$ is an open dense subset of $G$,
we have $\exp(tX) \in N_-MAN_+$ for sufficiently 
small $t >0$. Thus, we have
\begin{align*}
X
&=\frac{d}{dt}\bigg\vert_{t=0}\exp(tX)\\
&=\frac{d}{dt}\bigg\vert_{t=0}p_-(\exp(tX))p_0(\exp(tX))p_+(\exp(tX))\\
&=
\frac{d}{dt}\bigg\vert_{t=0}p_-(\exp(tX))
+
\frac{d}{dt}\bigg\vert_{t=0}p_0(\exp(tX))
+
\frac{d}{dt}\bigg\vert_{t=0}p_+(\exp(tX)),
\end{align*}
which implies
\begin{alignat*}{3}
&\frac{d}{dt}\bigg\vert_{t=0}\exp(tX_{\fn_-})&&=X_{\fn_-}&&=
\frac{d}{dt}\bigg\vert_{t=0}p_-(\exp(tX)),\\
&\frac{d}{dt}\bigg\vert_{t=0}\exp(tX_{\fl})&&=X_{\fl}&&=
\frac{d}{dt}\bigg\vert_{t=0}p_0(\exp(tX)),\\
&\frac{d}{dt}\bigg\vert_{t=0}\exp(tX_{\fn_+})&&=X_{\fn_+}&&=
\frac{d}{dt}\bigg\vert_{t=0}p_+(\exp(tX)).
\end{alignat*}
Further, for $\bar{n} \in N_-$ and sufficiently small $t >0$, we have 
\begin{equation*}
\exp(-tX)\bar{n} = p_-(\exp(-tX)\bar{n})p_0(\exp(-tX)\bar{n})p_+(\exp(-tX)\bar{n}).
\end{equation*}
Observe that
\begin{align*}
&p_-(\exp(-tX)\bar{n})
=p_-(\bar{n}\exp(-t\Ad(\bar{n}^{-1})X))
=\bar{n}p_-(\exp(-t\Ad(\bar{n}^{-1})X)),\\
&p_0(\exp(-tX)\bar{n})
=p_0(\bar{n}\exp(-t\Ad(\bar{n}^{-1})X))
=p_0(\exp(-t\Ad(\bar{n}^{-1})X)).
\end{align*}

\noindent
Then, for $F \in C^\infty(G/P,\Cal{V})\simeq C^\infty(G,V)^P$, we have
\begin{align*}
F(\exp(-tX)\bar{n})
&=
F(p_-(\exp(-tX)\bar{n})p_0(\exp(-tX)\bar{n}))\\
&=F(\bar{n}p_-(\exp(-t\Ad(\bar{n}^{-1})X))p_0(\exp(-t\Ad(\bar{n}^{-1})X)))\\
&=
\sigma_{\lambda}(p_0(\exp(-t\Ad(\bar{n}^{-1})X))^{-1})
F(\bar{n}p_-(\exp(-t\Ad(\bar{n}^{-1})X))).
\end{align*}
\noindent
We remark that the $N_+$-invariance property that 
$F(gn) = F(g)$ for $n \in N_+$ is applied in the first line.
Since 
\begin{alignat*}{2}
&\frac{d}{dt}\bigg\vert_{t=0}p_0
(\exp(-t\Ad(\bar{n}^{-1})X))^{-1}&&=(\Ad(\bar{n}^{-1})X)_\fl,\\
&\frac{d}{dt}\bigg\vert_{t=0}p_-(\exp(-t\Ad(\bar{n}^{-1})X))&&=
\frac{d}{dt}\bigg\vert_{t=0}\exp(-t(\Ad(\bar{n}^{-1})X)_{\fn_-}),
\end{alignat*}
the representation $\dpi_{(\sigma,\lambda)}(X)$ on
the image $\iota^*(C^\infty(G/P,\Cal{V}))$ of the inclusion 
$\iota^*\colon 
C^\infty(G/P, \Cal{V}) \hookrightarrow C^\infty(N_-)\otimes V$
 is given by
\begin{align}\label{eqn:dpi1}
\dpi_{(\sigma,\lambda)}(X)F(\bar{n})
&=\frac{d}{dt}\bigg\vert_{t=0}F(\exp(-tX)\bar{n}) \nonumber\\
&=\frac{d}{dt}\bigg\vert_{t=0}\sigma_{\lambda}(p_0(\exp(-t\Ad(\bar{n}^{-1})X))^{-1})
F(\bar{n}p_-(\exp(-t\Ad(\bar{n}^{-1})X))) \nonumber\\
&=
\frac{d}{dt}\bigg\vert_{t=0}
\sigma_{\lambda}(p_0(\exp(-t\Ad(\bar{n}^{-1})X))^{-1})F(\bar{n})
+
\frac{d}{dt}\bigg\vert_{t=0}
F(\bar{n}p_-(\exp(-t\Ad(\bar{n}^{-1})X))) \nonumber\\[3pt]
&=d\sigma_\lambda((\Ad(\bar{n}^{-1})X)_\fl)F(\bar{n})
-\left(dR((\Ad(\cdot^{-1})X)_{\fn_-})F\right)(\bar{n}).
\end{align}

Equation \eqref{eqn:dpi1} shows that,  for $X \in \fg(\R)$, 
the formula $\dpi_{(\sigma,\lambda)}(X)$ on $\iota^*(C^\infty(G/P,\Cal{V}))$ 
can be extended to the whole space $C^\infty(N_-)\otimes V$. 
By extending the formula complex linearly to $\fg$, we have
\begin{equation}\label{eqn:dpi2}
d\pi_{(\sigma,\lambda)}(X)f(\bar{n})
=d\sigma_\lambda((\Ad(\bar{n}^{-1})X)_\fl)f(\bar{n})
-\left(dR((\Ad(\cdot^{-1})X)_{\fn_-})f\right)(\bar{n})
\end{equation}
for $X \in \fg$ and $f(\bar{n}) \in C^\infty(N_-)\otimes V$.

\subsection{Duality theorem}\label{sec:duality}

For a finite-dimensional irreducible representation $(\sigma_\lambda,V)$ of $MA$, 
we write $V^\vee = \Hom_\C(V,\C)$ and $((\sigma_\lambda)^\vee, V^\vee)$ for
the contragredient representation of $(\sigma_\lambda,V)$. By letting $\fn_+$ act 
on $V^\vee$ trivially, we regard the infinitesimal representation 
$d\sigma^\vee \boxtimes \C_{-\lambda}$ of $(\sigma_\lambda)^\vee$ as a $\fp$-module. The induced module
\begin{equation*}
M_\fp(V^\vee) := \Cal{U}(\fg)\otimes_{\Cal{U}(\fp)}V^\vee
\end{equation*}
is called a generalized Verma module.
Via the diagonal action of $P$ on $M_\fp(V^\vee)$, we regard
$M_\fp(V^\vee)$ as a $(\fg, P)$-module.

The following theorem is often called the \emph{duality theorem}. For the proof, see, 
for instance, \cite{CS90, KP1, KR00}.

\begin{thm}[duality theorem]\label{thm:duality}
There is a natural linear isomorphism
\begin{equation}\label{eqn:duality1}
\EuScript{D}_{H\to D}
\colon
\operatorname{Hom}_{P}(W^\vee,M_{\fp}(V^\vee))
\stackrel{\sim}{\To} 
\operatorname{Diff}_{G}(\mathcal V, \mathcal W),
\end{equation}
where,
for $\varphi \in \Hom_P(W^\vee, M_\fp(V^\vee))$ and
$F \in C^\infty(G/P,\Cal{V})\simeq C^\infty(G,V)^P$,
the element $\EuScript{D}_{H\to D}(\varphi)F \in C^\infty(G/P,\Cal{W})\simeq C^\infty(G,W)^P$
is given by 
\begin{equation}\label{eqn:HD}
\IP{\EuScript{D}_{H\to D}(\varphi)F}{w^\vee} =
\sum_{j}\IP{dR(u_j)F}{v_j^\vee}
\;\; \text{for $w^\vee \in W^\vee$},
\end{equation}
where $\varphi(w^\vee)=\sum_j u_j\otimes v_j^\vee \in 
M_\fp(V^\vee)$.
\end{thm}

\begin{rem}
By the Frobenius reciprocity, \eqref{eqn:duality1} is equivalent to
\begin{equation}\label{eqn:duality2}
\EuScript{D}_{H\to D}
\colon
\operatorname{Hom}_{\fg, P}(M_\fp(W^\vee),M_{\fp}(V^\vee))
\stackrel{\sim}{\To} 
\operatorname{Diff}_{G}(\mathcal V, \mathcal W).
\end{equation}

\end{rem}

\begin{rem}\label{rem:Hom}
In general $P$ is not connected. 
If it is connected, then 
\begin{equation*}
\operatorname{Hom}_{\fg,P}(M_\fp(W^\vee),M_{\fp}(V^\vee))
=\operatorname{Hom}_{\fg}(M_\fp(W^\vee),M_{\fp}(V^\vee)).
\end{equation*}
Thus, in the case, the isomorphism \eqref{eqn:duality2} is equivalent to
\begin{equation}\label{eqn:duality3}
\EuScript{D}_{H\to D}\colon
\operatorname{Hom}_{\fg}(M_\fp(W^\vee),M_{\fp}(V^\vee))
\stackrel{\sim}{\To} 
\operatorname{Diff}_{G}(\mathcal V, \mathcal W).
\end{equation}
\end{rem}

\subsection{Algebraic Fourier transform $\widehat{\;\;\cdot\;\;}$ of Weyl algebras}
\label{sec:Weyl}

Let $U$ be a complex finite-dimensional vector space with $\dim _\C U=n$.
Fix a basis $u_1,\ldots, u_n$ of $U$ and let 
$(z_1, \ldots, z_n)$ denote the coordinates of $U$ with respect to the basis.
Then the algebra
\begin{equation*}
\C[U;z, \tfrac{\partial}{\partial z}]:=
\C[z_1, \ldots, z_n, 
\tfrac{\partial}{\partial z_1}, \ldots, \tfrac{\partial}{\partial z_n}]
\end{equation*}
with relations $z_iz_j = z_jz_i$, 
$\frac{\partial}{\partial z_i}\frac{\partial}{\partial z_j} 
=\frac{\partial}{\partial z_j}\frac{\partial}{\partial z_i}$,
and $\frac{\partial}{\partial z_j} z_i=\delta_{i,j} + z_i \frac{\partial}{\partial z_j}$
is called the Weyl algebra of $U$, where $\delta_{i,j}$ is the Kronecker delta.
Similarly,
let $(\zeta_1,\ldots, \zeta_n)$ denote the coordinates of 
the dual space $U^\vee$ of $U$ with respect to the dual basis of 
$u_1,\ldots, u_n$. We write 
$\C[U^\vee;\zeta, \tfrac{\partial}{\partial \zeta}]$
for the Weyl algebra of $U^\vee$.
Then the map determined by
\begin{equation}\label{eqn:Weyl}
\widehat{\frac{\partial}{\partial z_i}}:= -\zeta_i, \quad
\widehat{z_i}:=\frac{\partial}{\partial \zeta_i}
\end{equation}
gives a Weyl algebra isomorphism 
\begin{equation}\label{eqn:Weyl2}
\widehat{\;\;\cdot\;\;}\;\colon\C[U;z, \tfrac{\partial}{\partial z}] 
\stackrel{\sim}{\To} 
\C[U^\vee;\zeta, \tfrac{\partial}{\partial \zeta}], 
\quad T \mapsto \widehat{T}.
\end{equation}
\vskip 0.1in
\noindent
The map \eqref{eqn:Weyl2} is called
the \emph{algebraic Fourier transform of Weyl algebras}
(\cite[Def.\ 3.1]{KP1}).
We remark that the minus sign for $ -\zeta_i$ in \eqref{eqn:Weyl} is put in such a way that the resulting map in \eqref{eqn:Weyl2} is indeed a Weyl algebra isomorphism.
We remark that, for the Euler homogeneity operators 
$E_z = \sum_{i=1}^n z_i\frac{\partial}{\partial z_i}$ for $z$ 
and 
$E_\zeta = \sum_{i=1}^n \zeta_i\frac{\partial}{\partial \zeta_i}$ for $\zeta$,
we have $\widehat{E_z} = -(n+E_\zeta)$.

The Weyl algebra $\C[U;z, \tfrac{\partial}{\partial z}]$
naturally acts on the space $\Cal{D}'(U)$ of distributions on $U$.
In particular, the action of $\C[U;z, \tfrac{\partial}{\partial z}]$ preserves
the subspace $\Cal{D}'_{0}(U)$ of distributions supported at $0$.
Let $\Cal{J}$ be the annihilator of the Dirac delta function $\delta_0$, that is,
the kernel of the homomorphism
\begin{equation*}
\Psi\colon
\C[U;z, \tfrac{\partial}{\partial z}]
\To
\Cal{D}'_{0}(U), \quad P\mapsto P\delta_0.
\end{equation*}
Then $\Cal{J}$ is 
the left ideal of $\C[U;z, \tfrac{\partial}{\partial z}]$
generated by the coordinate functions $z_1,\ldots,z_n$.
Since the map $\Psi$ is surjective
(see, for instance, \cite[Thm.\ 5.5]{vanDijk13}),
it induces the isomorphism
\begin{equation}\label{eqn:DW}
\bar{\Psi}\colon
\C[U;z, \tfrac{\partial}{\partial z}]/\Cal{J}
 \stackrel{\sim}{\To} 
\Cal{D}'_{0}(U).
\end{equation}
On the other hand,
by applying the algebraic Fourier transform $\widehat{\;\;\cdot\;\;}$ 
in \eqref{eqn:Weyl2}
to $\C[U;z, \tfrac{\partial}{\partial z}]/\Cal{J}$,
the space $\C[U;z, \tfrac{\partial}{\partial z}]/\Cal{J}$
is isomorphic to the space $\Pol(U^\vee)=\C[U^\vee;\zeta]$ of polynomials on 
$U^\vee$. Thus, we have
\begin{equation}\label{eqn:DP}
\Cal{D}'_{0}(U) 
\stackrel[\sim]{\bar{\Psi}^{-1}}{\To} 
\C[U;z, \tfrac{\partial}{\partial z}]/\Cal{J}
\stackrel[\sim]{\widehat{\;\;\cdot\;\;}}{\To} 
\Pol(U^\vee).
\end{equation}
It is remarked that
the composition $\widehat{\;\;\cdot\;\;} \circ \bar{\Psi}^{-1}$ maps
$\Cal{D}'_{0}(U) \ni \delta_0 \mapsto 1 \in \Pol(U^\vee)$.

\subsection{Fourier transformed representation $\widehat{d\pi_{(\sigma,\lambda)^*}}$}
\label{sec:dpi2}

For $2\rho\equiv 2\rho(\fn_+)= \mathrm{Trace}(\ad\vert_{\fn_+})\in \mathfrak{a}^*$,
we denote by  $\C_{2\rho}$ the one-dimensional representation of $P$
defined by
$p \mapsto \chi_{2\rho}(p)=
\abs{\mathrm{det}(\mathrm{Ad}(p)\colon 
\mathfrak{n}_+ \to \mathfrak{n}_+)}$.
For the contragredient representation
$((\sigma_\lambda)^\vee, V^\vee)$ of $(\sigma_\lambda,V)$,
we put
$\sigma^*_\lambda := \sigma^\vee \boxtimes \C_{2\rho-\lambda}$.
As for $\sigma_\lambda$, we regard $\sigma^*_\lambda$ as a representation of $P$.
Define the induced representation
$\pi_{(\sigma, \lambda)^*} = \mathrm{Ind}_P^G(\sigma^*_\lambda)$
on the space $C^\infty(G/P,\mathcal{V}^*)$ of smooth sections 
for the vector bundle $\mathcal{V}^*=G\times_P (V^\vee\otimes \C_{2\rho})$
associated with $\sigma^*_\lambda$, which is isomorphic to the tensor bundle
of the dual vector bundle $\mathcal{V}^\vee = G\times_PV^\vee$ and the bundle
of densities over $G/P$.
Then the integration on $G/P$ gives a 
$G$-invariant non-degenerate bilinear form
$
\mathrm{Ind}^G_P(\sigma_\lambda) 
\times \mathrm{Ind}^G_P(\sigma^*_\lambda) \to \C
$
for $\mathrm{Ind}^G_P(\sigma_\lambda)$ and 
$\mathrm{Ind}^G_P(\sigma^*_\lambda)$.

As for $C^\infty(G/P,\mathcal{V})$,
the space $C^\infty(G/P,\mathcal{V}^*)$ can be regarded as 
a subspace of $C^\infty(N_-) \otimes V^\vee$.
Then, by replacing $\sigma_\lambda$ with $\sigma_\lambda^*$ in 
\eqref{eqn:dpi2}, we have
\begin{equation}\label{eqn:dpi3}
d\pi_{(\sigma,\lambda)^*}(X)f(\bar{n})
=d\sigma_\lambda^*((\Ad(\bar{n}^{-1})X)_\fl)f(\bar{n})
-\left(dR((\Ad(\cdot^{-1})X)_{\fn_-})f\right)(\bar{n})
\end{equation}
for $X \in \fg$ and $f(\bar{n}) \in C^\infty(N_-)\otimes V^\vee$.
Via the exponential map $\exp\colon \fn_-(\R) \simeq N_-$, one can regard
$\dpi_{(\sigma,\lambda)^*}(X)$ as a representation 
on $C^\infty(\mathfrak{n}_-(\R)) \otimes V^\vee$.
It then follows from \eqref{eqn:dpi3} that 
$\dpi_{(\sigma,\lambda)^*}$ gives a Lie algebra homomorphism
\begin{equation*}
d\pi_{(\sigma,\lambda)^*}\colon \mathfrak{g} \To 
\C[\fn_-(\R);x, \tfrac{\partial}{\partial x}]  \otimes \mathrm{End}(V^\vee),
\end{equation*}
where $(x_1, \ldots, x_n)$ are coordinates of $\fn_-(\R)$ with $n=\dim \fn_-(\R)$.
We extend the coordinate functions $x_1,\ldots, x_n$ for $\fn_-(\R)$ 
holomorphically to the ones $z_1, \ldots, z_n$ for $\fn_-$.
Thus we have 
\begin{equation*}
d\pi_{(\sigma,\lambda)^*}\colon \mathfrak{g} \To 
\C[\fn_-;z, \tfrac{\partial}{\partial z}]  \otimes \mathrm{End}(V^\vee).
\end{equation*}

Now we fix a non-degenerate symmetric bilinear form $\kappa$ on
$\fn_+$ and $\fn_-$. Via $\kappa$, we identify 
$\fn_+$  with the dual space $\fn_-^\vee$ of 
$\fn_-$. Then
the algebraic Fourier transform $\widehat{\;\;\cdot\;\;}$
of Weyl algebras \eqref{eqn:Weyl2} gives a Weyl algebra isomorphism
\begin{equation*}
\widehat{\;\;\cdot\;\;}\;\colon
\C[\fn_-;z, \tfrac{\partial}{\partial z}] 
\stackrel{\sim}{\To} 
\C[\fn_+;\zeta, \tfrac{\partial}{\partial \zeta}].
\end{equation*}
In particular, it gives a Lie algebra homomorphism
\begin{equation}\label{eqn:hdpi}
\widehat{d\pi_{(\sigma,\lambda)^*}}\colon \mathfrak{g} \To 
\C[\fn_+;\zeta, \tfrac{\partial}{\partial \zeta}]\otimes \mathrm{End}(V^\vee).
\end{equation}

\subsection{Algebraic Fourier transform $F_c$ of generalized Verma modules}
\label{sec:Fc}

Our aim is to construct a $(\fg, P)$-module isomorpshism
\begin{equation*}
M_\fp(V^\vee)
\stackrel{\sim}{\To} 
\Pol(\fn_+)\otimes V^{\vee},
\end{equation*}
where $\Pol(\fn_+)$ denotes the space 
$\C[\fn_+;\zeta]$ of polynomial functions on $\fn_+$.
It is classically known that the map
$u\otimes v^\vee \mapsto 
 dL(u)(\delta_{[o]} \otimes v^\vee)$ gives
a $(\fg, P)$-module isomorphism
\begin{equation}\label{eqn:VD}
M_\fp(V^\vee)
\stackrel{\sim}{\To} 
 \Cal{D}'_{[o]}(G/P,\Cal{V}^*),
 \quad
 u\otimes v^\vee \mapsto 
 dL(u)(\delta_{[o]} \otimes v^\vee),
\end{equation}
where $[o]$ denotes the identity coset $[o] = eP$ of $G/P$
and $\delta_{[o]}$ denotes the Dirac delta function supported at $[o]$
(see, for instance, \cite{KP1,Kostant75}).
It thus suffices to construct a $(\fg, P)$-module isomorphism
\begin{equation*}
 \Cal{D}'_{[o]}(G/P,\Cal{V}^*)
\stackrel{\sim}{\To} 
\Pol(\fn_+)\otimes V^{\vee}.
\end{equation*}

First, observe that the restriction of the dual vector bundle
$\mathcal{V}^*\to G/P$ to the open Bruhat cell
$\iota\colon \mathfrak{n}_-(\R) \simeq N_-\hookrightarrow G/P$
gives an isomorphism
\begin{equation*}
\iota^*\colon
 \Cal{D}'_{[o]}(G/P,\Cal{V}^*) 
 \stackrel{\sim}{\To} 
 \Cal{D}'_{0}(\fn_-(\R),V^\vee).
\end{equation*}
Via the isomorphism $\iota^*$,
one can induce a $(\fg,P)$-module structure on
$\Cal{D}'_{0}(\fn_-(\R),V^\vee)$ from 
$ \Cal{D}'_{[o]}(G/P,\Cal{V}^*)$.
In particular, the Lie algebra $\fg$ acts 
on $ \Cal{D}'_{0}(\fn_-(\R),V^\vee)$ via 
the infinitesimal representation $d\pi_{(\sigma,\lambda)^*}$.
Consequently, by the isomorphism \eqref{eqn:DP}, one obtains
the following chain of $(\fg,P)$-isomorphisms:
\begin{alignat}{4}\label{eqn:VP}
M_\fp(V^\vee)
&\stackrel{\sim}{\To} \Cal{D}'_{[o]}(G/P,\Cal{V}^*) 
&&\stackrel{\sim}{\To} \Cal{D}'_{0}(\fn_-(\R),V^\vee)
&&\stackrel{\sim}{\To} \Pol(\fn_+)\otimes V^\vee\\
\rotatebox{90}{$\in$}\hspace{20pt}
&\hspace{50pt}\rotatebox{90}{$\in$}
&&\hspace{50pt}\rotatebox{90}{$\in$}
&&\hspace{55pt}\rotatebox{90}{$\in$}& \nonumber\\[-2pt]
u\otimes v^\vee \hspace{8pt}
&\mapsto dL(u)(\delta_{[o]} \otimes v^\vee) 
&&\mapsto d\pi_{(\sigma,\lambda)^*}(u)(\delta_0\otimes v^\vee) 
&&\mapsto \widehat{d\pi_{(\sigma,\lambda)^*}}(u)(1\otimes v^\vee). \nonumber
\end{alignat}

We denote by $F_c$ the composition of the three 
$(\mathfrak{g}, P)$-module isomorphisms in \eqref{eqn:VP}.

\begin{defn}[{\cite[Sect.\ 3.4]{KP1}}]\label{def:AFT2}
We call the $(\mathfrak{g}, P)$-module isomorphism
\begin{equation}\label{eqn:Fc}
F_c\colon M_\fp(V^\vee) 
\stackrel{\sim}{\To}
\Pol(\fn_+) \otimes V^\vee,
\quad u\otimes v^\vee \longmapsto 
\widehat{d\pi_{(\sigma,\lambda)^*}}(u)(1\otimes v^\vee)\\
\end{equation}
the \emph{algebraic Fourier transform of the generalized Verma module  $M_\fp(V^\vee)$}.
\end{defn}

\subsection{The F-method}
\label{sec:Fmethod2}

Observe that
the algebraic Fourier transform
$F_c$  in \eqref{eqn:Fc} 
gives an $MA$-representation isomorphism
\begin{equation}\label{eqn:VP2}
M_\fp(V^\vee)^{\fn_+}
\stackrel{\sim}{\To}
(\Pol(\fn_+) 
\otimes V^\vee)^{\widehat{d\pi_{(\sigma,\lambda)^*}}(\fn_+)},
\end{equation}
which induces a linear isomorphism
\begin{equation}\label{eqn:VP31}
\operatorname{Hom}_{MA}(W^\vee,
M_\fp(V^\vee)^{\fn_+})
\stackrel{\sim}{\To}
\operatorname{Hom}_{MA}
\left(W^\vee,
(\mathrm{Pol}(\mathfrak{n}_+) 
\otimes V^\vee)^{\widehat{d\pi_{(\sigma,\lambda)^*}}(\fn_+)}\right).
\end{equation}
Here $MA$ acts on $\Pol(\fn_+)$ via the action
\begin{equation}\label{eqn:sharp}
\Ad_{\#}(l) \colon p(X) \mapsto p(\Ad(l^{-1})X)
\;\; \text{for $l \in MA$}.
\end{equation}

Recall from \eqref{eqn:hdpi} that we have
\begin{equation*}
\widehat{d\pi_{(\sigma,\lambda)^*}}(\fn_+)\subset
\C[\fn_+;\zeta, \tfrac{\partial}{\partial \zeta}]\otimes \mathrm{End}(V^\vee).
\end{equation*}
Since the nilpotent radical 
$\fn_+$ acts trivially on $1\otimes v^\vee \in M_\fp(V^\vee)$,
and since the algebraic Fourier transform $F_c$ is a $\fg$-module isomorphism,
the Fourier image $\widehat{d\pi_{(\sigma,\lambda)^*}}(\fn_+)$ must act on
$1\otimes v^\vee \in \Pol(\fn_+) \otimes V^\vee$ trivially.
Therefore,  the operator
$\widehat{d\pi_{(\sigma,\lambda)^*}}(C)$ for $C\in \fn_+$
is a differential operator of order at least 1;
in particular, the space
$(\mathrm{Pol}(\mathfrak{n}_+) 
\otimes V^\vee)^{\widehat{d\pi_{(\sigma,\lambda)^*}}(\fn_+)}$
is the space of polynomial solutions to the system of partial differential equations
defined by $\widehat{d\pi_{(\sigma,\lambda)^*}}(C)$ for $C \in \fn_+$.
It is remarked that the operators
$\widehat{d\pi_{(\sigma,\lambda)^*}}(C)$ are not vector fields;
for instance, if $\mathfrak{n}_+$ is abelian, then
$\widehat{d\pi_{(\sigma,\lambda)^*}}(C)$ are second order differential operators
(cf.\ Proposition \ref{prop:wdpi} and \cite[Prop.~3.10]{KP1}).

To emphasize that
$(\mathrm{Pol}(\mathfrak{n}_+) 
\otimes V^\vee)^{\widehat{d\pi_{(\sigma,\lambda)^*}}(\fn_+)}$
is the space of polynomial solutions to a system of differential equations,
we let
\begin{equation}\label{eqn:Sol2a}
\mathrm{Sol}(\mathfrak n_+;V,W)=
\operatorname{Hom}_{MA}(W^\vee,
(\mathrm{Pol}(\mathfrak{n}_+) 
\otimes V^\vee)^{\widehat{d\pi_{(\sigma,\lambda)^*}}(\fn_+)}).
\end{equation}
Via the identification
$\textrm{Hom}_{MA}(W^\vee, \Pol(\fn_+)\otimes V^\vee)
\simeq
\big((\Pol(\fn_+)\otimes V^\vee) \otimes W\big)^{MA}$,
we have
\begin{align}\label{eqn:Sol}
&\Sol(\fn_+;V,W)\nonumber \\[3pt]
&=\{ \psi \in \Hom_{MA}(W^\vee, \Pol(\fn_+)\otimes V^\vee): 
\text{
$\psi$ satisfies the system of PDEs \eqref{eqn:Fsys} below.}\}
\end{align}
\begin{equation}\label{eqn:Fsys}
(\widehat{d\pi_{(\sigma,\lambda)^*}}(C) \otimes \id_W)\psi=0
\,\,
\text{for all $C \in \fn_+$}.
\end{equation}
We call the system of PDEs \eqref{eqn:Fsys} the \emph{F-system}
(\cite[Fact 3.3 (3)]{KKP16}).
Since
\begin{equation*}
\operatorname{Hom}_{P}(W^\vee,M_\fp(V^\vee))
=
\operatorname{Hom}_{MA}(W^\vee,M_\fp(V^\vee)^{\fn_+}),
\end{equation*}
the isomorphism \eqref{eqn:VP31} together with \eqref{eqn:Sol2a} shows the following.

\begin{thm}
[F-method, {\cite[Thm.\ 4.1]{KP1}}]
\label{thm:Fmethod}
There exists a linear isomorphism
\begin{equation*}
F_c \otimes \mathrm{id}_{W}\colon 
\operatorname{Hom}_{P}(W^\vee,
M_\fp(V^\vee))
\stackrel{\sim}{\To}
\mathrm{Sol}(\mathfrak n_+;V,W).
\end{equation*}
Equivalently, we have
\begin{equation*}
\mathrm{Hom}_{\mathfrak g,P}(M_\fp(W^\vee),M_\fp(V^\vee))\stackrel{\sim}{\To}
\mathrm{Sol}(\mathfrak n_+;V,W).
\end{equation*}
\end{thm}

\vskip 0.1in

The diagram \eqref{eqn:isom2} below is the refinement of 
\eqref{eqn:SHD2} in the introduction.
\begin{equation}\label{eqn:isom2}
\xymatrix@R-=0.7pc@C-=0.5cm{
{}
& 
\mathrm{Sol}(\mathfrak n_+;V,W)
\ar@{.>}_{\sim}[rddd]
&{} \\
&&&\\
&\\
\operatorname{Hom}_{\mathfrak g,P}(M_\fp(W^\vee),M_\fp(V^\vee)) 
  \ar[rr]^{\hspace{25pt}\sim}_{\hspace{25pt}\EuScript{D}_{H\to D}}
  \ar[ruuu]_{\sim}^{F_c\otimes\mathrm{id}_W} 
 & {}
 & 
 \operatorname{Diff}_{G}(\mathcal V, \mathcal W)
}
\end{equation}

\subsection{The case of abelian nilradical $\fn_+$}
\label{sec:abelian}

Now suppose that the nilpotent radical $\fn_+$ is abelian.
In this case intertwining differential operators 
$\D \in \Diff_G(\Cal{V},\Cal{W})$ have constant coefficients.
Indeed, observe that, as $\fn_+$ is abelian, so is $\fn_-(\R)$. 
Thus, we have
\begin{equation*}
\exp(X)\exp(Y)=\exp(X+Y) \quad \text{for all $X, Y \in \fn_-(\R)$}.
\end{equation*}
For a given ordered basis $(X_1, \ldots, X_n)$ with
$n = \dim \fn_-(\R)$, one may identify $C^\infty(N_-)$ with
$C^\infty(N_-) \simeq C^\infty(\R^n)$ via 
the exponential map
\begin{equation}\label{eqn:20240503a}
\R^n \stackrel{\sim}{\To} N_-, \quad (x_1, \ldots, x_{n}) 
\mapsto \exp(x_1 X_1 + \cdots +  x_{n} X_{n}).
\end{equation}
Then, for $f(x_1, \ldots, x_n) \in C^\infty(\R^n)$, we have
\begin{align*}
dR(X_j) f(x_1, \ldots, x_n) 
&=\frac{d}{dt}\bigg\vert_{t=0}
f\big(\exp(x_1 X_1 + \cdots +  x_{n} X_{n})
\exp(tX_j)\big)\\
&=\frac{d}{dt}\bigg\vert_{t=0}
f\big(\exp(x_1 X_1 + \cdots + (x_j + t)X_j +\cdots+ x_{n} X_{n})\big)\\
&=\frac{d}{dt}\bigg\vert_{t=0}
f(x_1, \ldots, x_{j-1}, x_j + t, x_{j+1}, \ldots, x_n)\\
&=\frac{\partial}{\partial x_j}f(x_1, \ldots, x_n).
\end{align*}
Since intertwining differential operators $\D$
are given by linear combinations of the differential operators
$dR(X_1)^{k_1}\cdots dR(X_n)^{k_n}$ over $\C$ by Theorem \ref{thm:duality}, 
this shows that the coefficients of $\D$ must be constant.

As in \eqref{eqn:DN}, one may view $\Diff_G(\Cal{V},\Cal{W})$ as
\begin{equation*}
\Diff_G(\Cal{V},\Cal{W}) \subset \C[\fn_-;\frac{\partial}{\partial z}]
 \otimes \mathrm{Hom}_{\C}(V,W).
\end{equation*}
Since $\fn_+$ is regarded as the dual space of $\fn_-$, 
one can define the symbol map
\begin{equation*}
\mathrm{symb}\colon
\C[\fn_-;\frac{\partial}{\partial z}]
\To\C[\fn_+;\zeta],
\quad 
\frac{\partial}{\partial z_i} \mapsto \zeta_i.
\end{equation*}
As $\Pol(\fn_+) = \C[\fn_+;\zeta]$, this shows that we have
\begin{align}\label{eqn:const2}
\mathrm{symb}\colon
\C[\fn_-;\frac{\partial}{\partial z}]
 \otimes \mathrm{Hom}_{\C}(V,W) 
\stackrel{\sim}{\To} &\;
\Pol(\fn_+)
\otimes \mathrm{Hom}_{\C}(V,W) \nonumber \\
\text{{\footnotesize{$\bigcup$}} \hspace{2.8cm}}& 
\hspace{1.48cm} 
\text{{\footnotesize{$\bigcup$}}} \\
\textnormal{Diff}_{G}(\mathcal{V}, \mathcal{W}) \hspace{2cm} 
&
\hspace{-0.75cm}\stackrel{\sim}{\To} 
\hspace{0.5cm} 
\mathrm{symb}(\textnormal{Diff}_{G}(\mathcal{V}, \mathcal{W}))
 \hspace{0.5cm}\nonumber
\end{align}
On the other hand, by \eqref{eqn:Sol2a}, we also have
\begin{equation*}
\mathrm{Sol}(\mathfrak n_+;V,W)
\simeq
\big(\mathrm{Pol}(\mathfrak{n}_+) 
\otimes \operatorname{Hom}_{\C}(V,W)\big)^{MA,\, \widehat{d\pi_{(\sigma,\lambda)^*}}(\fn_+)}.
\end{equation*}
Hence,
\begin{equation*}
\xymatrix@R=-1pt@C=-35pt{
& 
\mathrm{Pol}(\mathfrak{n}_+) \otimes \operatorname{Hom}_{\C}(V,W)
&\\
\rotatebox[origin=c]{-45}{$\bigcup$}
\hspace{5pt}
&
&
\rotatebox[origin=c]{45}{$\bigcup$}
\hspace{20pt}
\\
\mathrm{Sol}(\mathfrak n_+;V,W) \hspace{10pt}
&
&
\mathrm{symb}(\textnormal{Diff}_{G}(\mathcal{V}, \mathcal{W}))
}
\end{equation*}
\noindent
Theorem \ref{thm:abelian} below shows that, in fact, we have
\begin{equation*}
\mathrm{Sol}(\mathfrak n_+;V,W)
= 
\mathrm{symb}(\textnormal{Diff}_{G}(\mathcal{V}, \mathcal{W})).
\end{equation*}

\begin{thm}
[{\cite[Cor.\ 4.3]{KP1}}]
\label{thm:abelian}
Suppose that the nilpotent radical
$\mathfrak{n}_+$ is abelian. Then 
the symbol map $\mathrm{symb}$ gives
a linear isomorphism
\begin{equation*}
\mathrm{symb}^{-1}
\colon 
\mathrm{Sol}(\mathfrak n_+;V,W)
\stackrel{\sim}{\To} 
 \operatorname{Diff}_{G}(\mathcal V, \mathcal W).
\end{equation*}
Further, the diagram \eqref{eqn:isom3} below commutes.
\begin{equation}\label{eqn:isom3}
\xymatrix@R-=0.7pc@C-=0.5cm{
{}
& 
\mathrm{Sol}(\mathfrak n_+;V,W)
\ar[rddd]_{\sim}^{\;\; \mathrm{symb}^{-1}}
&{} \\
&&&\\
&\circlearrowleft{} \\
\operatorname{Hom}_{\mathfrak g,P}(M_\fp(W^\vee),M_\fp(V^\vee)) 
  \ar[rr]^{\hspace{25pt}\sim}_{\hspace{25pt}\EuScript{D}_{H\to D}}
  \ar[ruuu]_{\sim}^{F_c\otimes\mathrm{id}_W} 
 & {}
 & 
 \operatorname{Diff}_{G}(\mathcal V, \mathcal W)
}
\end{equation}
\end{thm}

\subsection{A recipe of the F-method for abelian nilradical $\fn_+$}
\label{sec:recipe}

By \eqref{eqn:Sol} and Theorem \ref{thm:abelian},
one can classify and construct 
$\D \in  \operatorname{Diff}_{G}(\mathcal V, \mathcal W)$ and
$\varphi \in \operatorname{Hom}_{\mathfrak g,P}(M_\fp(W^\vee),M_\fp(V^\vee))$
by computing $\psi \in \mathrm{Sol}(\mathfrak n_+;V,W)$ as follows.
\vskip 0.1in

\begin{enumerate}

\item[Step 1]
Compute $d\pi_{(\sigma,\lambda)^*}(C)$ and 
$\widehat{d\pi_{(\sigma,\lambda)^*}}(C)$
for $C \in \fn_+$.
\vskip 0.1in

\item[Step 2]
Classify and construct 
$\psi \in \Hom_{MA}(W^\vee, \Pol(\fn_+)\otimes V^\vee)$.
\vskip 0.1in

\item[Step 3]
Solve the F-system \eqref{eqn:Fsys}
for $\psi \in \Hom_{MA}(W^\vee, \Pol(\fn_+)\otimes V^\vee)$.
\vskip 0.1in

\item[Step 4]
For $\psi \in \mathrm{Sol}(\mathfrak n_+;V,W)$ obtained in Step 3,
do the following.
\vskip 0.1in

\begin{enumerate}
\item[Step 4a]
Apply $\symb^{-1}$ 
to $\psi \in \mathrm{Sol}(\mathfrak n_+;V,W)$
to obtain $\D \in  \operatorname{Diff}_{G}(\mathcal V, \mathcal W)$.
\vskip 0.1in
\item[Step 4b]
Apply $F_c^{-1} \otimes \id_W$ 
to $\psi \in \mathrm{Sol}(\mathfrak n_+;V,W)$
to obtain $\varphi \in \operatorname{Hom}_{\mathfrak g,P}(M_\fp(W^\vee),M_\fp(V^\vee))$.
\end{enumerate}
\vskip 0.1in

\end{enumerate}

\vskip 0.1in

In Section \ref{sec:proof}, we shall carry out the recipe for the case 
$(G,P) = (SL(n,\R), P_{1,n-1})$.

\section{Specialization to $(SL(n,\R), P_{1,n-1})$}\label{sec:SLn}

The aim of this short section is to specialize the general framework in 
Section \ref{sec:Fmethod} to the case  $(G, P) = (SL(n,\R), P_{1,n-1})$,
where $P_{1,n-1}$ denotes the maximal parabolic subgroup of $G$ corresponding
to the partition $n= 1 + (n-1)$. Throughout this section, we assume $n\geq 2$.

\subsection{Notation}\label{sec:notation}

Let $G = SL(n,\R)$ with Lie algebra $\fg(\R)=\f{sl}(n,\R)$ for $n\geq 2$.
We put
\begin{equation*}
N_j^+:=E_{1,j+1},
\quad
N_j^-:=E_{j+1,1}
\quad
\text{for $j\in \{1,\ldots, n-1\}$}
\end{equation*}
and 
\begin{equation*}
H_0:=\frac{1}{n}((n-1)E_{1,1} - \sum^n_{r=2}E_{r,r})
=\frac{1}{n}\diag(n-1, -1, -1, \ldots, -1),
\end{equation*}
where $E_{i,j}$ denote the matrix units. We normalize $H_0$ as
$\wH_0:=\frac{n}{n-1}H_0$,
namely,
\begin{align*}
\wH_0=\frac{1}{n-1}((n-1)E_{1,1} - \sum^n_{r=2}E_{r,r})
=\frac{1}{n-1}\diag(n-1, -1, -1, \ldots, -1).
\end{align*}
Let
\begin{align}\label{eqn:nR}
\fn_+(\R) =\Ker(\ad(H_0)-\id) 
&= \Ker(\ad(\wH_0)-\tfrac{n}{n-1}\id),\\
\fn_-(\R)=\Ker(\ad(H_0)+\id) 
&=\Ker(\ad(\wH_0)+\tfrac{n}{n-1}\id).\nonumber
\end{align}
Then we have
\begin{equation*}
\fn_{\pm}(\R)=\spn_{\R}\{N_1^{\pm},\ldots, N_{n-1}^{\pm} \}.
\end{equation*}

For $X, Y \in \fg(\R)$, let $\Tr(X,Y)=\text{Trace}(XY)$ denote the trace form of $\fg(\R)$. 
Then $N_i^+$ and $N_j^-$ satisfy $\Tr(N_i^+,N_j^-)=\delta_{i,j}$.
In what follows, we identify  the dual $\fn_-(\R)^\vee$ of $\fn_-(\R)$
with $\fn_-(\R)^\vee \simeq \fn_+(\R)$ via the trace form $\Tr(\cdot, \cdot)$.

Let $\fa(\R)= \R \wH_0$ and
\begin{equation}\label{eqn:Lm}
\fm(\R)= \left\{
\begin{pmatrix}
0 & \\
 & X
\end{pmatrix}
:
X \in \f{sl}(n-1,\R)
\right\}\simeq \f{sl}(n-1,\R).
\end{equation}
Here $\f{sl}(1,\R)$ is regarded as $\f{sl}(1,\R)=\{0\}$. 
We have $\fm(\R)\oplus \fa(\R) = \Ker(\ad(H_0))$ and
the decomposition
$\fg(\R) = \fn_-(\R) \oplus \fm(\R) \oplus \fa(\R) \oplus \fn_+(\R)$
is a Gelfand--Naimark decomposition of $\fg(\R)$.
The subalgebra
$\fp(\R):=\fm(\R) \oplus \fa(\R) \oplus \fn_+(\R)$ is a maximal parabolic 
subalgebra of $\fg(\R)$. 
It is remarked that $\fn_{\pm}(\R)$ are abelian.

Let $P$ be  the normalizer $N_G(\fp(\R))$ of $\fp(\R)$ in $G$. 
We write $P=MAN_+$ for the Langlands decomposition of $P$ corresponding to 
$\fp(\R)=\fm(\R) \oplus \fa(\R) \oplus \fn_+(\R)$. 
Then $A = \exp(\fa(\R)) = \exp(\R \wH_0)$ and 
$N_+=\exp(\fn_+(\R))$. The group $M$ is given by
\begin{equation*}
M= 
\left\{
\begin{pmatrix}
\det(g)^{-1} &\\
& g\\
\end{pmatrix}
:
g \in SL^{\pm}(n-1,\R)
\right\}\simeq SL^{\pm}(n-1,\R).
\end{equation*}
Here $SL^{\pm}(1,\R)$ is regarded as
$SL^{\pm}(1,\R) = \{\pm 1\}$.
As $M$ is not connected, let $M_0$ denote the identity component of $M$.
We write
\begin{equation*}
\gamma = \diag(-1, 1, \ldots, 1, -1).
\end{equation*}
Then $M_0 \simeq SL(n-1,\R)$ and 
\begin{equation*}
M/M_0 = \{[I_n], [\gamma]\} \simeq \Z/2\Z,
\end{equation*}
where $I_n$ is the $n \times n$ identity matrix and $[g] = g M_0$.

For a closed subgroup $J$ of $G$, we denote by $\Irr(J)$ and $\Irr(J)_{\fin}$
the sets of equivalence classes of irreducible representations of $J$  and 
finite-dimensional irreducible representations of $J$, respectively.

For $\lambda \in \C$, 
we define a one-dimensional representation 
$\C_\lambda=(\chi^\lambda, \C)$ of $A=\exp(\R\wH_0)$ by
\begin{equation}\label{eqn:chi}
\chi^\lambda \colon \exp(t \wH_0) \longmapsto \exp(\lambda t).
\end{equation}
Then $\Irr(A)$ is given by
\begin{equation*}
\Irr(A)=\{\C_\lambda : \lambda \in \C\} \simeq \C.
\end{equation*}

For $\alpha \in \{\pm\}$, a one-dimensional representation $\C_\alpha$ of $M$ 
is defined  by
\begin{equation*}
\begin{pmatrix}
\det(g)^{-1} &\\
& g\\
\end{pmatrix}
\longmapsto
\sgn^\alpha(\det(g)),
\end{equation*}
where
\begin{equation*}
\sgn^\alpha(\det(g)) 
= 
\begin{cases}
1 & \text{if $\alpha=+$},\\
\sgn(\det(g)) & \text{if $\alpha=-$}.
\end{cases}
\end{equation*}
Then $\Irr(M)_\fin = \Irr(SL^{\pm}(n-1,\R))_\fin$ is given by
\begin{equation*}
\Irr(M)_{\fin}=
\{\C_\alpha \otimes \varpi: 
(\alpha, \varpi) \in \{\pm\} \times \Irr(SL(n-1,\R))_{\fin}\}.
\end{equation*}
Since $\Irr(P)_{\fin}\simeq \Irr(M)_{\fin} \times \Irr(A)$, the set $\Irr(P)_\fin$ can be 
parametrized by
\begin{equation*}
\Irr(P)_{\fin}\simeq 
\{\pm\} \times \Irr(SL(n-1,\R))_{\fin} \times \C.
\end{equation*}

For $(\ga, \varpi, \lambda) \in \{\pm\} \times \Irr(SL(n-1,\R))_{\fin} \times \C$,
we write
\begin{equation}\label{eqn:Ind}
I(\varpi,\lambda)^\ga 
= 
\Ind_{P}^G\left((\C_\ga \otimes \varpi)\boxtimes \C_\lambda\right)
\end{equation}
for unnormalized parabolically induced representation 
$\Ind_{P}^G\left((\C_\ga \otimes \varpi)\boxtimes \C_\lambda \right)$ 
of $G$.
For instance, the unitary axis of $I(\triv, \lambda)^\ga$ is $\text{Re}(\lambda) = \frac{n}{2}$, where $\triv$ denotes the trivial representation of $SL(n-1,\R)$.
Similarly, for the representation space $W$ of $(\ga, \varpi, \lambda) \in \Irr(P)_{\fin}$,
we write 
\begin{equation}\label{eqn:Verma}
M_\fp(\varpi,\lambda)^\ga 
=\Cal{U}(\fg) \otimes_{\Cal{U}(\fp)}W.
\end{equation}

In the next section we classify and construct intertwining differential operators 
\begin{equation*}
\D \in \Diff_G(I(\triv, \lambda)^\ga, I(\varpi,  \nu)^\beta)
\end{equation*}
and $(\fg, P)$-homomorphisms
\begin{equation*}
\varphi \in \Hom_{\fg, P}(M_\fp(\varpi,\nu)^\ga, M_\fp(\triv, \lambda)^\gb).
\end{equation*}

\section{Classification and construction of $\D$ and $\varphi$}
\label{sec:Dphi}

The aim of this section is to state the classification and construction results 
for intertwining differential operators 
$\D \in \Diff_G(I(\triv, \lambda)^\ga, I(\varpi,  \nu)^\beta)$
as well as $(\fg, P)$-homomorphisms 
$\varphi \in \Hom_{\fg, P}(M_\fp(\varpi,\nu)^\ga, M_\fp(\triv, \lambda)^\gb)$.
These are given in 
Theorems \ref{thm:IDO1} and \ref{thm:IDO} for $\D$ 
and Theorems \ref{thm:Hom1} and \ref{thm:Hom} for $\varphi$.
The proofs of the theorems will be discussed in Section \ref{sec:proof}.
We remark that 
the case of $n=3$ is studied in \cite{Kubo22}.

\subsection{Classification and construction of intertwining differential operators $\D$}
\label{sec:D}
We start with the classification of the parameters 
 $(\alpha,\beta; \varpi; \lambda, \nu) 
\in \{\pm\}^2\times \Irr(SL(n-1,\R))_{\fin} \times \C^2$
such that $\Diff_G(I(\triv, \lambda)^\alpha, I(\varpi, \nu)^\beta) \neq \{0\}$.

For $\ga \in \{\pm\}\equiv \{\pm 1\}$ 
and $k \in \Z_{\geq 0}$, we mean 
$\ga + k \in \{\pm\}$ by
\begin{equation*}
\ga+k = 
\begin{cases}
+ & \text{if $\ga = (-1)^k$},\\
- & \text{if $\ga = (-1)^{k+1}$}.
\end{cases}
\end{equation*}
Then we define $\Lambda^n_{G} \subset \{\pm\}^2 \times \Irr(SL(n-1,\R))_{\fin} \times \C^2$ 
as
\begin{equation}\label{eqn:20240506b}
\Lambda^n_{G}=\{
(\alpha, \alpha+k;\poly_{n-1}^k; 1-k, 
1+\tfrac{k}{n-1} ): \alpha \in \{\pm\} \; \text{and} \; k\in \Z_{\geq 0}\},
\end{equation}
where $\poly_{n-1}^k$ denotes the irreducible representation 
on $\Pol^k(\C^{n-1}) = S^k((\C^{n-1})^\vee)$ 
of $SL(n-1,\R)$ induced by the standard action
on $\C^{n-1}$. We regard $\poly_1^k$ as $\poly_1^k = \triv$ for all $k \in \Z_{\geq 0}$.

\begin{thm}
\label{thm:IDO1}
The following three conditions on
 $(\alpha,\beta; \varpi; \lambda, \nu) 
\in \{\pm\}^2\times \Irr(SL(n-1,\R))_{\fin} \times \C^2 $ are 
equivalent.
\begin{enumerate}
\item[\emph{(i)}] 
$\Diff_G(I(\triv, \lambda)^\alpha, I(\varpi, \nu)^\beta) \neq \{0\}$.
\item[\emph{(ii)}] $\dim \Diff_G(I(\triv, \lambda)^\alpha, I(\varpi, \nu)^\beta) =1 $.
\item[\emph{(iii)}] 
One of the following two conditions holds:
\begin{enumerate}
\item[\emph{(iii-a)}]
$(\beta,\varpi, \nu) = (\alpha, \triv, \lambda)$.
\item[\emph{(iii-b)}]
$(\alpha,\beta; \varpi; \lambda, \nu)  \in \Lambda^n_G$.
\end{enumerate}
\end{enumerate}
\end{thm}

We next consider the explicit formula of 
$\D \in \Diff_G(I(\triv, \lambda)^\ga, I(\varpi, \nu)^\gb)$
for $(\alpha,\beta; \varpi; \lambda, \nu)$ in (iii) of Theorem \ref{thm:IDO1}.
We write
\begin{equation*}
\Pol^k(\C^{n-1}) = \C^k[y_1,\ldots, y_{n-1}].
\end{equation*}
Then, 
as in \eqref{eqn:DN} and \eqref{eqn:20240503a},
we understand 
$\D \in \Diff_G(I(\triv, \lambda)^\ga, I(\varpi,\nu)^\gb)$ 
for $(\alpha,\beta; \varpi; \lambda, \nu) \in \gL^n_G$
as a map
\begin{equation*}
\D\colon C^\infty(\R^{n-1}) \To 
C^{\infty}(\R^{n-1})\otimes \C^k[y_1,\ldots, y_{n-1}]
\end{equation*}
via the diffeomorphism
\begin{equation}\label{eqn:n-}
\R^{n-1} \stackrel{\sim}{\To} N_-, \quad (x_1, \ldots, x_{n-1}) 
\mapsto \exp(x_1 X_1 + \cdots +  x_{n-1} X_{n-1}).
\end{equation}

For $k \in \Z_{\geq 0}$, we put
\begin{equation*}
\Xi_k:=\{(k_1, \ldots, k_{n-1}) \in (\Z_{\geq 0})^{n-1} : \sum_{j=1}^{n-1} k_j= k\}.
\end{equation*}

\noindent
For $\mathbf{k} = (k_1, \ldots, k_{n-1})\in \Xi_k$, we write
\begin{align*}
\widetilde{y}_{\mathbf{k}} &= \frac{1}{k_1! \cdots k_{n-1}!} 
\cdot
y_1^{k_1}\cdots y_{n-1}^{k_{n-1}},\\[3pt]
\frac{\partial^k}{\partial x^{\mathbf{k}}} \
&= \frac{\partial^k}{\partial x_1^{k_1} \cdots \partial x_{n-1}^{k_{n-1}}}.
\end{align*}

\noindent
We define $\D_k \in \Diff_\C(C^\infty(\R^{n-1}), C^\infty(\R^{n-1})\otimes 
\C^k[y_1, \ldots, y_{n-1}])$ for $k \in \Z_{\geq 0}$ by
\begin{equation}\label{eqn:IDO}
\D_k = \sum_{\mathbf{k} \in \Xi_k} \frac{\partial^k}{\partial x^{\mathbf{k}}}
\otimes 
\widetilde{y}_{\mathbf{k}}.
\end{equation}

\noindent
For $k=0$, we understand $\D_0$ as the identity operator $\D_0 = \id$.

\begin{thm}\label{thm:IDO}
We have
\begin{equation*}
\Diff_G(I(\triv, \lambda)^\alpha, I(\varpi, \nu)^\beta)
=
\begin{cases}
\C\id & \text{if $(\beta, \varpi, \nu) = (\alpha, \triv, \lambda)$,}\\
\C \D_k & \text{if $(\alpha, \beta;\varpi; \lambda, \nu)\in \Lambda^n_{G}$,}\\
\{0\} & \text{otherwise.}
\end{cases}
\end{equation*}
\end{thm}

\subsection{Classification and construction of $(\fg, P)$-homomorphisms $\varphi$}
\label{sec:phi}

Let $\fg =\fg(\R)\otimes_{\R}\C= \f{sl}(n,\C)$.
We regard $\f{sl}(1,\C)$ as $\f{sl}(1,\C) = \{0\}$. 

Define $\Lambda^n_{(\fg, P)} \subset 
\{\pm\}^2\times \Irr(SL(n-1,\R))_{\fin} \times \C^2$ as
\begin{equation*}
\Lambda^n_{(\fg, P)}=\{
(\ga,\ga+k;\sym_{n-1}^k; k-1,  -(1+\tfrac{k}{n-1})) : \alpha \in \{\pm\} \; \text{and} \; k\in \Z_{\geq 0}\},
\end{equation*}
where $\sym_{n-1}^k$ denotes the irreducible representation 
on $S^k(\C^{n-1})$ of $SL(n-1,\R)$.
As for $\poly_1^k$, we regard $\sym_1^k$ 
as $\sym_1^k = \triv$ for all $k \in \Z_{\geq 0}$.

The classification of the parameters 
 $(\alpha,\beta; \sigma; s, r) \in 
\{\pm\}^2\times \Irr(SL(n-1,\R))_{\fin} \times \C^2$
such that $\Hom_{\fg,P}(M_\fp(\sigma,r)^\gb, M_\fp(\triv, s)^\ga)\neq \{0\}$
is given as follows.

\begin{thm}
\label{thm:Hom1}
The following three conditions on 
 $(\alpha,\beta; \sigma; s, r) \in 
\{\pm\}^2\times \Irr(SL(n-1,\R))_{\fin} \times \C^2$
are 
equivalent.
\begin{enumerate}
\item[\emph{(i)}] 
$\Hom_{\fg,P}(M_\fp(\sigma,r)^\gb, M_\fp(\triv, s)^\ga)\neq \{0\}$.
\item[\emph{(ii)}] 
$\dim \Hom_{\fg, P}(M_\fp(\sigma,r)^\gb, M_\fp(\triv, s)^\ga) =1 $.
\item[\emph{(iii)}] 
One of the following two conditions holds:
\begin{enumerate}
\item[\emph{(iii-a)}]
$(\gb,\sigma, r) = (\ga,\triv, s)$.
\item[\emph{(iii-b)}]
$(\ga, \gb; \sigma; s, r)  \in \Lambda^n_{(\fg, P)}$.
\end{enumerate}
\end{enumerate}
\end{thm}

To give the explicit formula of 
$\varphi \in \Hom_{\fg,P}(M_\fp(\sigma,r)^\gb, M_\fp(\triv, s)^\ga)$,
we write
\begin{equation*}
S^k(\C^{n-1}) = \C^k[e_1, \ldots, e_{n-1}],
\end{equation*}
where $e_j$ are the standard basis elements of $\C^{n-1}$. 

For $\mathbf{k} = (k_1, \ldots, k_{n-1})\in \Xi_k$, we write
\begin{alignat*}{2}
e_{\mathbf{k}} &= e_1^{k_1} \cdots e_{n-1}^{k_{n-1}} &&\in S^k(\C^{n-1}),\\
N_{\mathbf{k}}^- &= (N_1^-)^{k_1} \cdots (N_{n-1}^-)^{k_{n-1}}
&&\in S^k(\fn_-).
\end{alignat*}
\noindent
Observe that we have 
\begin{equation*}
\C^k[y_1, \ldots, y_{n-1}] 
= \Pol^k(\C^{n-1}) = S^k((\C^{n-1})^\vee)\simeq S^k(\C^{n-1})^\vee.
\end{equation*}
We then define $y_j\in (\C^{n-1})^\vee$ in such a way that  $y_i(e_j) = \delta_{i,j}$,
which gives 
$\widetilde{y}_{\mathbf{k}}(e_{\mathbf{k}'})=\delta_{\mathbf{k},\mathbf{k}'}$
for $\mathbf{k},\mathbf{k}' \in \Xi_k$.

We define $\varphi_k \in 
\Hom_\C(S^k(\C^{n-1}), S^k(\fn_-))$ by means of
\begin{equation}\label{eqn:Hom}
\varphi_k 
= \sum_{\mathbf{k} \in \Xi_k} 
N_{\mathbf{k}}^-
\otimes 
(e_{\mathbf{k}})^\vee
= \sum_{\mathbf{k} \in \Xi_k} 
N_{\mathbf{k}}^-
\otimes 
\widetilde{y}_{\mathbf{k}}.
\end{equation}

\noindent
Since $M_\fp(\triv, s)^\ga \simeq S(\fn_-)$ as linear spaces,
we have
\begin{equation*}
\varphi_k \in \Hom_{\C}(S^k(\C^{n-1}), M_\fp(\triv, s)^\ga).
\end{equation*}

\noindent 
Further, the following holds.

\begin{thm}\label{thm:Hom}
We have
\begin{equation*}
\Hom_{\fg, P}(M_\fp(\sigma,r)^\gb, M_\fp(\triv, s)^\ga)
=
\begin{cases}
\C\id & \text{if $(\gb, \sigma, r) = (\ga, \triv, s)$,}\\
\C \varphi_k & \text{if $(\ga, \gb, \sigma; s, r)\in \gL^n_{(\fg,P)}$,}\\
\{0\} & \text{otherwise.}
\end{cases}
\end{equation*}
\end{thm}

Here, by abuse of notation, we regard $\varphi_k$ as a map
\begin{equation*}
\varphi_k \in 
\Hom_{\fg, P}(
M_\fp(\sym^k_{n-1}, -(1+\tfrac{k}{n-1}))^{\ga+k},
M_\fp(\triv,k-1)^\ga)
\end{equation*}
defined by
\begin{equation*}
\varphi_k(u\otimes w) := u\varphi_k(w)
\quad 
\text{for $u \in \Cal{U}(\fg)$ and $w \in S^k(\C^{n-1})$.}
\end{equation*}

\begin{rem}
If $\Hom_{\fg, P}(M_\fp(\sigma,r)^\gb, M_\fp(\triv, s)^\ga) \neq 
\{0\}$, then the infinitesimal character of
$M_\fp(\triv, s)^\ga$ agrees with that of $M_\fp(\sigma,r)^\gb$. 
It follows from the arguments on the standardness 
of the homomorphism $\varphi_k$ in Section \ref{sec:Std}
that the infinitesimal characters 
are indeed equal 
for $(\ga, \gb; \sigma; s, r)  \in \Lambda^n_{(\fg, P)}$.
(See the proof of Theorem \ref{thm:std}.)
\end{rem}

\section{Proofs of the classification and construction}
\label{sec:proof}

The aim of this section is to give proofs of 
Theorems \ref{thm:IDO1} and \ref{thm:IDO} 
and Theorems \ref{thm:Hom1} and \ref{thm:Hom}.
As the nilpotent radical $\fn_+$ is abelian,
we achieve them simultaneously by proceeding with Steps 1--4 of the 
recipe of the F-method in Section \ref{sec:recipe}.

\subsection{Step 1:
Compute $d\pi_{(\sigma,\lambda)^*}(C)$ and 
$\widehat{d\pi_{(\sigma,\lambda)^*}}(C)$ for $C \in \fn_+$}
\label{sec:step1}

For $\sigma = \alpha \otimes \triv$ and $\lambda \equiv \chi^\lambda$, 
we simply write
\begin{equation*}
\dpi_{\lambda^*} = \dpi_{(\alpha\otimes\triv,\chi^\lambda)^*}
\end{equation*}
with $\lambda^*=2\rho(\fn_+)-\lambda d\chi$. 
We put
$E_x = \sum_{j=1}^{n-1} x_j\frac{\partial}{\partial x_j}$ for
the Euler homogeneity operator for $x$.

\begin{prop}\label{prop:dNj1}
For $N_j^+ \in \fn_+$, we have
\begin{equation}\label{eqn:dNj1}
d\pi_{\lambda^*}(N_j^+)=
x_j\{ (n-\lambda) 
+E_x\}.
\end{equation} 
\end{prop}

\begin{proof}
It follows from \eqref{eqn:dpi3} that $d\pi_{\lambda^*}(N_j^+)$ is given by
\begin{equation}
d\pi_{\lambda^*}(N_j^+)f(\bar{n})
=\lambda^*((\Ad(\bar{n}^{-1})N_j^+)_\fl)f(\bar{n})
-\big(dR((\Ad(\cdot^{-1})N_j^+)_{\fn_-})f\big)(\bar{n}).
\end{equation}
A direct computation shows that
\begin{align*}
(\Ad(\bar{n}^{-1})N_j^+)_\fl
&=x_j(E_{1,1}-E_{j+1,j+1})-\sum_{\substack{r=1\\r\neq j}}^{n-1}x_r N_r^-,\\
-(\Ad(\bar{n}^{-1})N_j^+)_{\fn_-}
&=x_j\sum_{r=1}^{n-1}x_rN_r^-.
\end{align*}

\noindent
Since $\lambda^*(E_{1,1}-E_{j+1,j+1}) = n-\lambda$ 
and $dR(N_r^-) = \frac{\partial}{\partial x_r}$ via 
the diffeomorphism \eqref{eqn:n-}, this shows the proposition.
\end{proof}

For later convenience, we next give the formula for $-\zeta_j\widehat{\dpi_{\lambda^*}}(N_j^+)$ instead of $\widehat{\dpi_{\lambda^*}}(N_j^+)$.
In the following, we silently extend the coordinate functions
$x_1, \ldots, x_{n-1}$ on $\fn_-(\R)$ in \eqref{eqn:dNj1}
holomorphically to the ones 
$z_1,\ldots, z_{n-1}$ on $\fn_-$ as in Section \ref{sec:dpi2}.

For $j \in \{1,\ldots, n-1\}$,
we write $\vartheta_j = \zeta_j \frac{\partial}{\partial \zeta_j}$ 
for the Euler operator for $\zeta_j$.
We also write $E_\zeta = \sum_{j=1}^{n-1} \vartheta_j$ for 
the Euler homogeneity operator for $\zeta$.
Observe that we have
$\widehat{E_z}=-((n-1)+E_\zeta)$.

\begin{prop}\label{prop:wdpi}
For $N_j^+ \in \fn_+$, we have
\begin{equation}\label{eqn:dNj2}
-\zeta_j
\widehat{\dpi_{\lambda^*}}(N_j^+)
=\vartheta_j(\lambda -1 + E_\zeta).
\end{equation}
\end{prop}

\begin{proof}
This  follows from a direct application of 
the algebraic Fourier transform \eqref{eqn:Weyl} of Weyl algebras 
to \eqref{eqn:dNj1}.
\end{proof}

As $E_\zeta \vert_{\Pol^k(\fn_+)} = k\cdot \id$,
the operator
$-\zeta_j\widehat{\dpi_{\lambda^*}}(N_j^+)\vert_{\Pol^k(\fn_+)}$
restricted to $\Pol^k(\fn_+)$ is given by
\begin{equation}\label{eqn:dNj3}
-\zeta_j\widehat{\dpi_{\lambda^*}}(N_j^+)\vert_{\Pol^k(\fn_+)}
=(\lambda-1+k)\vartheta_j.
\end{equation}
In Step 3 (Section \ref{sec:step3}), we use \eqref{eqn:dNj3} to solve the F-system
in concern.

\begin{rem}\label{rem:Fmethod}
In general,
the Fourier transformed operator 
$\widehat{\dpi_{\lambda^*}}(N_j^+)$
is not necessarily reduced to a first order operator 
as in \eqref{eqn:dNj3}. In fact, in \cite{KP2, KOSS15}, such operators give rise to
the Jacobi differential equation and Gegenbauer differential equation.
\end{rem}

\subsection{Step 2: Classify and construct 
$\psi \in \Hom_{MA}(W^\vee, \Pol(\fn_+)\otimes V^\vee)$}
\label{sec:step2}

For $(\varpi, W) \in \Irr(M_0)_{\fin}$, we write
\begin{equation*}
W_{\beta} = \C_{\beta} \otimes W
\end{equation*}
for 
the representation
space of $(\beta, \varpi) \in \Irr(M)_{\fin}$. 
Then, in this step, we wish to classify and construct
\begin{equation*}
\psi \in \Hom_{MA}(W_{\beta}^\vee\boxtimes \C_{-\nu}, 
\Pol^k(\fn_+)_{\alpha}
\otimes \C_{-\lambda}).
\end{equation*}

\subsubsection{Notation}
\label{sec:step2a}

We start by introducing some notation. 
For $\mathbf{k} \in \Xi_k$, we write
\begin{equation*}
\zeta_{\mathbf{k}}= \zeta_1^{k_1} \cdots \zeta_{n-1}^{k_{n-1}}
\in \Pol^k(\fn_+).
\end{equation*}
We then define 
$\psi_k \in \Hom_\C(S^k(\C^{n-1}),\Pol^k(\fn_+))$ by
\begin{equation}\label{eqn:psi}
\psi_k = 
\sum_{\mathbf{k}\in \Xi_k}
\zeta_{\mathbf{k}} \otimes \widetilde{y}_{\mathbf{k}},
\end{equation}
where $\widetilde{y}_{\mathbf{k}}\in \Pol^k(\C^{n-1})\simeq S^k(\C^{n-1})^\vee$ are regarded as 
the dual basis of $e_{\mathbf{k}} \in S^k(\C^{n-1})$.

Recall from \eqref{eqn:sharp} that $MA$ acts on 
$\Pol(\fn_+)$ via the action $\Ad_{\#}$.
In the present case, $(\Ad_{\#}, \Pol^k(\fn_+))$ 
is an irreducible representation of $M_0 \simeq SL(n-1,\R)$, which is
equivalent to
\begin{equation}\label{eqn:M01}
(M_0, \Ad_{\#}, \Pol^k(\fn_+))
\simeq (M_0,\Ad^k_{\fn_-},S^k(\fn_-))
\simeq (SL(n-1,\R), \sym_{n-1}^k, S^k(\C^{n-1})).
\end{equation}
Since $\psi_k$ maps
$\psi_k\colon e_{\mathbf{k}} \mapsto \zeta_{\mathbf{k}}$, we have
\begin{equation}\label{eqn:M0}
\psi_k \in \Hom_{M_0}(S^k(\C^{n-1}),\Pol^k(\fn_+)).
\end{equation}

\subsubsection{Classification and construction of 
$\psi \in \Hom_{MA}(W^\vee, \Pol(\fn_+)\otimes V^\vee)$}
\label{sec:step2b}

As 
\begin{equation*}
\Hom_{MA}(W_{\beta}^\vee\boxtimes \C_{-\nu}, 
\Pol^k(\fn_+)_{\alpha}\otimes \C_{-\lambda})
\subset \Hom_{M_0}(W^\vee, \Pol^k(\fn_+)),
\end{equation*}
we first consider $\Hom_{M_0}(W^\vee, \Pol^k(\fn_+))$.

Since $(\sym_{n-1}^k, S^k(\C^{n-1}))^\vee \simeq (\poly_{n-1}^k,\Pol^k(\C^{n-1}))$,
it follows from \eqref{eqn:M01} that
the following two conditions on a representation 
$W$ of $M_0 \simeq SL(n-1,\R)$ are equivalent.
\begin{enumerate}
\item[(i)]
$W^\vee \simeq (\Ad_{\#}, \Pol^k(\fn_+))$.

\item[(ii)]
$W \simeq (\poly_{n-1}^k, \Pol^k(\C^{n-1}))$.
\end{enumerate}

\noindent
Now $\Hom_{M_0}(W^\vee, \Pol^k(\fn_+))$ is given as follows.

\begin{prop}\label{prop:M0}
The following three conditions on 
a representation $W$ of $M_0 \simeq SL(n-1,\R)$
are equivalent.
\begin{enumerate}
\item[\emph{(i)}]
$\Hom_{M_0}(W^\vee, \Pol^k(\fn_+)) \neq \{0\}$.

\item[\emph{(ii)}]
$\dim \Hom_{M_0}(W^\vee, \Pol^k(\fn_+))=1$.

\item[\emph{(iii)}]
$W \simeq (\poly_{n-1}^k, \Pol^k(\C^{n-1}))$.

\end{enumerate}
Consequently, we have
\begin{equation}\label{eqn:M02}
\Hom_{M_0}(W^\vee, \Pol^k(\fn_+))=
\begin{cases}
\C \psi_k & \text{if $(\varpi, W)\simeq (\poly_{n-1}^k, \Pol^k(\C^{n-1}))$,}\\
\{0\} & \text{otherwise.}
\end{cases}
\end{equation}
\end{prop}

\begin{proof}
As the $M_0$-representation $(\Ad_{\#}, \Pol^k(\fn_+))$ is irreducible,
the first assertion follows from Schur's lemma and the preceding argument.
The second assertion \eqref{eqn:M02} is a simple consequence of
the first and \eqref{eqn:M0}.
\end{proof}

A direct computation shows that we have
\begin{equation*}\label{eqn:20240619}
(M,\Ad^k_{\fn_-},S^k(\fn_-))
\simeq (SL^{\pm}(n-1,\R), \sgn^{k}\otimes \sym_{n-1}^k, S^k(\C^{n-1})).
\end{equation*}
Here  we mean 
$\sgn^k$ by
$\sgn^k = 
\triv$ if $k \equiv 0 \Mod 2$;
$\sgn$ if $k \equiv 1 \Mod 2$.
Since $(M, \Ad_{\#}, \Pol^k(\fn_+)) \simeq (M,\Ad^k_{\fn_-},S^k(\fn_-))$,
it  implies
\begin{equation}\label{eqn:20240629}
(M, \Ad_{\#}, \Pol^k(\fn_+))
\simeq (SL^{\pm}(n-1,\R), \sgn^{k}\otimes \sym_{n-1}^k, S^k(\C^{n-1})).
\end{equation}

\begin{prop}\label{prop:sym}
We have
\begin{equation*}
\Hom_{MA}(W_{\beta}^\vee\boxtimes \C_{-\nu}, 
\Pol^k(\fn_+)_{\alpha}\otimes \C_{-\lambda})=
\begin{cases}
\C \psi_k & \text{if 
$(\alpha, \beta; \varpi; \lambda, \nu) = (\alpha, \alpha+k;
\poly_{n-1}^k;\lambda, \lambda+\frac{n}{n-1}k)$,}\\
\{0\} & \text{otherwise.}
\end{cases}
\end{equation*}
\end{prop}

\begin{proof}
It follows from \eqref{eqn:nR} that,
via the action $\Ad_{\#}$, 
the group $A=\exp(\R \wH_0)$ acts on $\Pol^k(\fn_+)$ by a character
$\chi^{-\frac{n}{n-1}k}$. 
Now Proposition \ref{prop:M0} and \eqref{eqn:20240629} conclude the assertion.
\end{proof}

In Proposition \ref{prop:sym}, the value of $\lambda \in \C$ is still arbitrary.
One approach to determine $\lambda \in \C$
for which 
$\Diff_G(I(\triv, \lambda)^\alpha, I(\varpi, \nu)^\beta) \neq \{0\}$
is to solve the square norm of the infinitesimal characters in question.
(See, for instance, \cite[Sect.\ 7]{BKZ09}.) In the following, we instead compute 
the $\fn_+$-invariance via the F-method to obtain such $\lambda \in \C$ more directly.

\subsection{Step 3: 
Solve the F-system 
for $\psi \in \Hom_{MA}(W^\vee, \Pol(\fn_+)\otimes V^\vee)$}
\label{sec:step3}

For $(\ga, \gb; \varpi; \lambda,\nu) \in \{\pm\}^2\times \Irr(SL(n-1,\R))_\fin \times 
\C^2$ and $k\in \Z_{\geq 0}$, we put
\begin{align*}
&\Sol^k(\fn_+; \triv_{\alpha, \lambda}, \varpi_{\beta, \nu})\\[3pt]
&:=
\{ \psi \in 
\Hom_{MA}(W_{\beta}^\vee\boxtimes \C_{-\nu}, 
\Pol^k(\fn_+)_{\alpha}\otimes \C_{-\lambda}): 
\text{
$\psi$ solves the F-system \eqref{eqn:Fsys2} below.}\}.
\end{align*}
\begin{equation}\label{eqn:Fsys2}
(\widehat{\dpi_{\lambda^*}}(N_j^+)\otimes \id_W)\psi =0 
\quad
\text{for all $j \in \{1,\ldots, n-1\}$}.
\end{equation}

Since
\begin{equation*}
\Sol^k(\fn_+; \triv_{\alpha, \lambda}, \varpi_{\beta, \nu})
\subset
\Hom_{MA}(W_{\beta}^\vee\boxtimes \C_{-\nu}, 
\Pol^k(\fn_+)_{\alpha}\otimes \C_{-\lambda}),
\end{equation*}
Proposition \ref{prop:sym} shows that
if $\Sol^k(\fn_+; \triv_{\alpha, \lambda}, \varpi_{\beta, \nu})\neq \{0\}$,
then $(\alpha, \beta; \varpi; \lambda, \nu)$ must satisfy
\begin{equation}\label{eqn:param}
(\alpha, \beta; \varpi; \lambda, \nu) = (\alpha, \alpha+k;
\poly_{n-1}^k;\lambda, \lambda+\frac{n}{n-1}k).
\end{equation}

\begin{thm}\label{thm:Sol1}
Let $(\alpha, \beta; \varpi; \lambda,\nu) \in \{\pm\}^2 \times \Irr(SL(n-1,\R))_{\fin}\times
\C^2$.
The following conditions on $(\alpha, \beta; \varpi; \lambda,\nu)$ are equivalent.
\begin{enumerate}
\item[\emph{(i)}] $\Sol^k(\fn_+; \triv_{\alpha, \lambda}, \varpi_{\beta, \nu}) \neq \{0\}$.
\item[\emph{(ii)}] One of the following two conditions holds.
\begin{enumerate}
\item[\emph{(ii-a)}]
$(\beta, \varpi, \nu) = (\alpha, \triv, \lambda)$.
\item[\emph{(ii-b)}]
$(\alpha, \beta;\varpi; \lambda, \nu) =  
(\alpha, \alpha+k;\poly_{n-1}^k; 1-k, 
1+\tfrac{k}{n-1}) $.
\end{enumerate}
\end{enumerate}
Further, the space $\Sol^k(\fn_+; \triv_{\alpha, \lambda}, \varpi_{\beta, \nu})$
is given as follows.
\begin{enumerate}
\item[\emph{(1)}] 
$(\beta, \varpi, \nu) = (\alpha, \triv, \lambda)$:
\begin{equation*}
\Sol^0(\fn_+; \triv_{\alpha, \lambda}, \triv_{\alpha, \lambda}) = \C.
\end{equation*}
\item[\emph{(2)}] 
$(\alpha, \beta;\varpi; \lambda, \nu) =  
(\alpha, \alpha+k;\poly_{n-1}^k; 1-k, 
1+\tfrac{k}{n-1}) $:
\begin{equation*}
\Sol^k(\fn_+; \triv_{\alpha, \lambda}, \varpi_{\beta, \nu})
=\C \psi_k,
\end{equation*}
where $\psi_k$ is the map defined in \eqref{eqn:psi}.
\end{enumerate}
\end{thm}

\begin{proof}
Observe that 
$\psi \in 
\Hom_{MA}(W_{\beta}^\vee\boxtimes \C_{-\nu}, 
\Pol^k(\fn_+)_{\alpha}\otimes \C_{-\lambda})$ solves
\eqref{eqn:Fsys2} if and only if it satisfies a system of PDEs
\begin{equation*}
(-\zeta_j\widehat{d\pi_{\lambda^*}}(N_j^+) \otimes \id_W)\psi_k=0
\quad
\text{for all $j\in \{1,\ldots, n-1\}$.}
\end{equation*}
By \eqref{eqn:psi}, we have
\begin{equation*}
(-\zeta_j\widehat{\dpi_{\lambda^*}}(N_j^+)\otimes \id_W)\psi_k 
=
\sum_{\mathbf{k}\in \Xi_k}
-\zeta_j\widehat{\dpi_{\lambda^*}}(N_j^+)(\zeta_{\mathbf{k}})
\otimes \widetilde{y}_{\mathbf{k}}.
\end{equation*}
So, one wishes to solve
\begin{equation}\label{eqn:dpiz}
-\zeta_j\widehat{\dpi_{\lambda^*}}(N_j^+)(\zeta_{\mathbf{k}})=0
\quad
\text{for all $j \in \{1,\ldots, n-1\}$ and $\mathbf{k} \in \Xi_k$}.
\end{equation}
It follows from \eqref{eqn:dNj3} that \eqref{eqn:dpiz} can be simplified to
\begin{equation}\label{eqn:theta}
(\lambda-1+k)\vartheta_j(\zeta_{\mathbf{k}}) =0
\quad
\text{for all $j \in \{1,\ldots, n-1\}$ and $\mathbf{k} \in \Xi_k$}.
\end{equation}
We have 
$(\lambda-1+k)\vartheta_j(\zeta_{\mathbf{k}}) 
=(\lambda-1+k)k_j\zeta_{\mathbf{k}}$.
Thus, \eqref{eqn:theta} holds if and only if $k=0$ or $\lambda = 1-k$.
Now the theorem follows from \eqref{eqn:param}.
\end{proof}

As for the $MA$-decomposition $\Pol(\fn_+) 
= \bigoplus_{k \in \Z_{\geq 0}} \Pol^k(\fn_+)$,
we put
\begin{equation}\label{eqn:20240506}
\Sol(\fn_+; \triv_{\alpha, \lambda}, \varpi_{\beta, \nu})
:=\bigoplus_{k \in \Z_{\geq 0}}
\Sol^k(\fn_+; \triv_{\alpha, \lambda}, \varpi_{\beta, \nu}).
\end{equation}
\noindent
Recall from \eqref{eqn:20240506b} that we have
\begin{equation*}
\Lambda^n_{G}=\{
(\alpha, \alpha+k;\poly_{n-1}^k; 1-k, 
1+\tfrac{k}{n-1} ): \alpha \in \{\pm\} \; \text{and} \; k\in \Z_{\geq 0}\}.
\end{equation*}
Corollary \ref{cor:Sol1} below is then a direct consequence of Theorem \ref{thm:Sol1}.

\begin{cor}\label{cor:Sol1}
For $(\alpha, \beta; \varpi; \lambda,\nu) \in \{\pm\}^2 \times \Irr(SL(n-1,\R))_{\fin}\times
\C^2$,
we have
\begin{equation*}
\Sol(\fn_+; \triv_{\alpha, \lambda}, \varpi_{\beta, \nu})
=
\begin{cases}
\C & \text{if $(\beta, \varpi, \nu) = (\alpha, \triv, \lambda)$,}\\
\C \psi_k & \text{if $(\alpha, \beta;\varpi; \lambda, \nu)\in \Lambda^n_{G}$,}\\
\{0\} & \text{otherwise.}
\end{cases}
\end{equation*}
Here $\psi_k$ is the map defined in \eqref{eqn:psi}.
\end{cor}

\subsection{Step 4ab:
Apply $\symb^{-1}$ and $F_c^{-1} \otimes \id_W$ to the solutions $\psi$ to the F-system}
\label{sec:step4}

Observe that 
$\D_k$ in \eqref{eqn:IDO} and $\varphi_k$ in \eqref{eqn:Hom} are 
given by
\begin{alignat*}{3}
\D_k 
&=\sum_{\mathbf{k} \in \Xi_k}
\frac{\partial^k}{\partial x^{\mathbf{k}}} \otimes \widetilde{y}_{\mathbf{k}}
&&=\sum_{\mathbf{k} \in \Xi_k}
\symb^{-1}(\zeta_{\mathbf{k}}) \otimes \widetilde{y}_{\mathbf{k}}
&&= \symb^{-1}(\psi_k),\\
\varphi_k 
&= 
\sum_{\mathbf{k} \in \Xi_k} 
N_{\mathbf{k}}^-
\otimes 
\widetilde{y}_{\mathbf{k}}
&&=\sum_{\mathbf{k} \in \Xi_k} 
F_c^{-1}(\zeta_\mathbf{k})
\otimes 
\widetilde{y}_{\mathbf{k}}
&&=(F_c^{-1} \otimes \id_{W})(\psi_k).
\end{alignat*}
\noindent
Now we are ready to prove Theorems \ref{thm:IDO1} and \ref{thm:IDO}
and Theorems \ref{thm:Hom1} and \ref{thm:Hom}.

\begin{proof}[Proofs of Theorems \ref{thm:IDO1}, \ref{thm:IDO}, 
\ref{thm:Hom1}, and \ref{thm:Hom}]
By Theorem \ref{thm:abelian}, we have
\begin{align*}
\Diff_G(I(\triv, \lambda)^\ga, I(\varpi,  \nu)^\beta)
&=
\symb^{-1}(\Sol(\fn_+; \triv_{\alpha, \lambda}, \varpi_{\beta, \nu})),\\
\Hom_{\fg, P}(M_\fp(\varpi^\vee,-\nu)^\ga, M_\fp(\triv, -\lambda)^\gb)
&=
(F_c^{-1}\otimes \id_{W})(\Sol(\fn_+; \triv_{\alpha, \lambda}, \varpi_{\beta, \nu})).
\end{align*}
Since $(\poly_{n-1}^k)^\vee \simeq \sym_{n-1}^k$, 
the theorems in consideration follow from Corollary \ref{cor:Sol1}.
\end{proof}

\subsection{Classification and construction of $\fg$-homomorphisms}
\label{sec:gHom}

Theorem \ref{thm:Hom} concerns $(\fg, P)$-homomorphisms between 
generalized Verma modules. 
Then we finish this section by showing the classification and construction of 
$\fg$-homomorphisms.

Let $P_0$ be the identity component of the parabolic subgroup $P$. 
Then we have $P_0 = M_0AN_+$. Thus, $\Irr(P_0)_\fin$ is given by
\begin{equation*}
\Irr(P_0)_\fin \simeq \Irr(M_0)_\fin \times \Irr(A) 
\simeq \Irr(\f{sl}(n-1,\C))_{\fin} \times \C \simeq \Irr(\fp)_{\fin}.
\end{equation*}

For $(\sigma, s) \in \Irr(\f{sl}(n-1,\C))_\fin\times \C$, 
we define a generalized Verma module $M_\fp(\sigma,s)$ as in \eqref{eqn:Verma}
as a $\fg$-module. For $(\varpi; \lambda, \nu) \in \Irr(\f{sl}(n-1,\C))_\fin\times \C^2$, 
we let
\begin{equation*}
\Sol(\fn_+; \triv_{\lambda}, \varpi_{\nu})
=
\{ \psi \in 
\Hom_{M_0A}(W^\vee\boxtimes \C_{-\nu}, 
\Pol(\fn_+)\otimes \C_{-\lambda}): 
\text{
$\psi$ solves the F-system \eqref{eqn:Fsys2}.}\}.
\end{equation*}
Then, by Remark \ref{rem:Hom}, we have
\begin{equation}\label{eqn:HD0}
F_c \otimes \id_W \colon
\Hom_\fg(M_\fp(\varpi^\vee, -\nu), M_\fp(\triv,-\lambda))
\stackrel{\sim}{\To}
\Sol(\fn_+; \triv_{\lambda}, \varpi_{\nu}).
\end{equation}

Define $\Lambda^n_{\fg} \subset 
\Irr(\f{sl}(n-1,\C))_\fin \times \C^2$ as
\begin{equation*}
\Lambda^n_{\fg}=\{
(\sym_{n-1}^k; k-1, -(1+\tfrac{k}{n-1})) :  k\in \Z_{\geq 0}\}.
\end{equation*}

\begin{thm}\label{thm:Hom2}
We have
\begin{equation*}
\Hom_{\fg}(M_\fp(\sigma,r), M_\fp(\triv, s))
=
\begin{cases}
\C\id & \text{if $(\sigma, r) = (\triv, s)$,}\\
\C \varphi_k & \text{if $(\sigma; s, r)\in \gL_{\fg}^n$,}\\
\{0\} & \text{otherwise.}
\end{cases}
\end{equation*}
\end{thm}

\begin{proof}
By \eqref{eqn:HD0}, it suffices to compute
$\Sol(\fn_+; \triv_{\lambda}, \varpi_{\nu})$.
Then the assertion simply follows from the same arguments in 
Step 1 (Section \ref{sec:step1}) -- Step 4 (Section \ref{sec:step4}).
\end{proof}

\begin{cor}\label{cor:Red}
The following are equivalent on $s \in \C$.
\begin{enumerate}
\item[\emph{(i)}] $M_\fp(\triv, s)$ is reducible.
\item[\emph{(ii)}] $s \in \Z_{\geq 0}$.
\end{enumerate}
\end{cor}

\begin{proof}
Observe that $M_\fp(\triv,s)$ is reducible if and only if there exists 
$(\sigma,r) \in \Irr(\fp)_\fin\backslash\{(\triv,s)\}$ 
such that $\Hom_{\fg}(M_\fp(\sigma,r), M_\fp(\triv,s)) \neq \{0\}$.
Now the assertion follows from Theorem \ref{thm:Hom2}.
\end{proof}

\begin{rem}
The results of Corollary \ref{cor:Red} is already known in the literature.
See, for instance,
\cite[Ex.\ 5.2]{BX21},
\cite[Thm.\ 1.1]{He15}, and
\cite[p.\ 794, Table 1]{HKZ19}.
\end{rem}

\section{$K$-type formulas for $\Ker(\D_k)^\ga$ and $\Im(\D_k)^\ga$}
\label{sec:Sol}

The aim of this section is to classify the $K$-type formulas of 
the kernel $\Ker(\D_k)^\ga$ and image $\Im(\D_k)^\ga$ of the 
non-zero intertwining differential operator
\begin{equation*}
\D_k \colon I(\triv, 1-k)^\ga \To I(\poly_{n-1}^k,1+\tfrac{k}{n-1})^{\ga+k}.
\end{equation*}
The $K$-type formulas are obtained in Corollary \ref{cor:IDO4}.
We continue the notation and normalization from Section \ref{sec:SLn},
unless otherwise specified.

Although the main idea in this section works 
for the case $n=2$,
to avoid the complication of the exposition,
we constraint ourselves to the
case $n\geq 3$.

\subsection{Composition factors and $K$-type structure of $I(\triv, \lambda)^\ga$}

Let $K=SO(n)$ be a maximal compact subgroup of $G=SL(n,\R)$. Then we have
\begin{align*}
K\cap M =
& S(O(1)\times O(n-1))\\
&=\left\{\begin{pmatrix}
\det(g)^{-1} & \\
 & g
\end{pmatrix}
: g \in O(n-1)
 \right\}\\[3pt]
&\simeq O(n-1).
\end{align*}
We remark that $K\cap M/(K\cap M)_0 \simeq M/M_0 \simeq \Z/2\Z$.

For a representation $V$ of $G$, we denote by $V_K$ the space of $K$-finite vectors of $V$. Since $G=KP$, as $K$-representations, we have
\begin{equation*}
I(\triv, \lambda)^\ga_K
\simeq 
\Ind_{K\cap M}^K(\alpha)_K.
\end{equation*}

Let  $\Cal{H}^m(\R^n)$ be the irreducible representation of $K$ consisting 
of spherical harmonics on $S^{n-1} \subset \R^n$ of 
homogeneous degree $m$ (cf.\ \cite[Sect.\ 7.5]{KM11}).
It is known that the $K$-type decomposition of 
$I(\triv, \lambda)^\ga_K\simeq \Ind_{K\cap M}^K(\alpha)_K$ is given as follows
(see, for instance, \cite[p.\ 286]{HL99}).
\begin{align}\label{eqn:K-type1}
I(\triv, \lambda)^\ga_K
&\simeq 
\bigoplus_{\stackrel{m \in \Z_{\geq 0}}{ \alpha = (-1)^m}}
\Cal{H}^m(\R^{n}) \nonumber \\[5pt]
&=
\begin{cases}
\bigoplus_{\ell \in \Z_{\geq 0}} \Cal{H}^{2\ell}(\R^{n}) & \text{if $\alpha =+$,}\\[5pt]
\bigoplus_{\ell \in \Z_{\geq 0}} \Cal{H}^{2\ell+1}(\R^{n}) & \text{if $\alpha =-$.}
\end{cases}
\end{align}

Theorem \ref{thm:vDM} below states well-known facts on 
the irreducibility and composition series of $I(\triv, \lambda)^\ga$.
For the proof, see, for instance,
\cite{HL99, vDM99} for $\alpha = \pm$ and \cite{MS14} for $\alpha =+$.

\begin{thm}\label{thm:vDM}
Let $n\geq 3$.
For $\ga \in \{\pm\} \equiv \{\pm 1 \}$ and $\lambda \in \C$, 
$I(\triv, \lambda)^\ga$ enjoys the following.

\begin{enumerate}
\item[\emph{(1)}]
The induced representation $I(\triv, \lambda)^\ga$ is irreducible except the following
two cases.
\vskip 0.05in
\begin{enumerate}
\item[\emph{(A)}]
$\lambda \in -\Z_{\geq 0}$ and $\alpha = (-1)^\lambda$.
\vskip 0.05in

\item[\emph{(B)}]
$\lambda \in n+\Z_{\geq 0}$ and $\alpha = (-1)^{\lambda+n}$.
\end{enumerate}

\vskip 0.1in

\item[\emph{(2)}]
For Case (A) with $\lambda = -m$, there exists 
a finite-dimensional irreducible subrepresentation
$F(-m)^\ga \subset I(\triv, -m)^\ga$ such that 
$I(\triv, -m)^\ga/F(-m)^\ga$ is irreducible and 
infinite-dimensional. The $K$-type formulas of $F(-m)^\ga_K=F(-m)^\ga$ 
are given as follows.
\begin{equation*}
F(-m)^+_K \simeq
\bigoplus_{\ell = 0}^{m/2}
\Cal{H}^{2\ell}(\R^n) 
\quad \text{and} \quad
F(-m)^-_K \simeq
\bigoplus_{\ell = 0}^{(m-1)/2}
\Cal{H}^{2\ell+1}(\R^n)
\end{equation*}

\item[\emph{(3)}]
For Case (B) with $\lambda =n+m$, 
there exists an infinite-dimensional
irreducible subrepresentation
$T(n+m)^\ga \subset I(\triv,n+m)^\ga$ such that 
$I(\triv, n+m)^\ga/T(n+m)^\ga$
is irreducible and finite-dimensional.
The $K$-type formulas of $T(n+m)^\ga_K$ are given as follows.
\begin{equation*}
T(n+m)^+_K \simeq
\bigoplus_{\ell \geq  \frac{m+2}{2}}
\Cal{H}^{2\ell}(\R^n) 
\quad \text{and} \quad
T(n+m)^-_K \simeq
\bigoplus_{\ell \geq \frac{m+1}{2}}
\Cal{H}^{2\ell+1}(\R^n)
\end{equation*}

\item[\emph{(4)}]
For $m \in \Z_{\geq 0}$,
we have the non-split exact sequences of Fr{\'e}chet $G$-modules:
\begin{alignat*}{4}
\{0\} &\To F(-m)^\ga &&\To I(\triv, -m)^\ga &&\To T(n+m)^\ga &&\To \{0\},\\
\{0\} &\To T(n+m)^\ga &&\To I(\triv, n+m)^\ga &&\To F(-m)^\ga &&\To \{0\}.
\end{alignat*}
\end{enumerate}
\end{thm}

\begin{rem}\label{rem:vDM1}
In \cite[Thm.\ 1.1]{vDM99}, the condition ``$\mu < 2-n$'' in b) should be read as 
``$\mu < 1-n$''. For $\mu=1-n$, the induced representation 
$\pi^{\pm}_{1-n,\nu}$
in the cited paper is irreducible, as it is dual to the case of $\mu = -1$.
\end{rem}

\begin{rem}\label{rem:vDM2}
The parabolic subgroup $P_{n-1,1}$ corresponding to the partition
$n=(n-1)+1$ is considered in \cite{HL99} and \cite{vDM99},
while we consider the one corresponding to $n=1+(n-1)$ in this paper.
Thus, the complex parameters ``$\mu$'' in \cite{vDM99}, ``$\ga$'' in \cite{HL99}, and
``$\lambda$'' in this paper are related as $\mu = \alpha = -\lambda$. 
\end{rem}

Now, for $k \in \Z_{\geq 0}$, 
we consider the $K$-type formulas of the kernel 
$\Ker(\D_k)^\alpha_K$ and image $\Im(D_k)^\alpha_K$ of the intertwining 
differential operator 
\begin{equation*}
\D_k \colon 
I(\triv, 1-k)^\ga \To I(\poly^k_{n-1}, 1+\tfrac{k}{n-1})^{\ga+k}.
\end{equation*}

If $k=0$, then 
$I(\triv, 1-k)^\ga = I(\poly^k_{n-1}, 1+\tfrac{k}{n-1})^\ga = I(\triv, 1)^\ga$
and $\D_0 = \id$. Thus, in this case, we have 
\begin{equation*}
\Ker(\D_0)^\alpha = \{0\}
\quad
\text{and}
\quad
\Im(\D_0)^\alpha = I(\triv, 1)^\ga.
\end{equation*}
\noindent
Therefore,
the $K$-type formula $\Im(\D_0)^\alpha_K$ is given as in \eqref{eqn:K-type1}.
For $k \in 1 + \Z_{\geq 0}$, we have the following.

\begin{thm}\label{thm:IDO3}
For 
$\ga \in \{\pm\} \equiv \{\pm 1\}$ and
$k \in 1 + \Z_{\geq 0}$,
the kernel $\Ker(\D_k)^\alpha$ is given as follows.
\begin{equation*}
\Ker(\D_k)^\alpha = 
\begin{cases}
F(1-k)^\ga & \text{if $\alpha = (-1)^{1-k}$,}\\[3pt]
\{0\} & \text{otherwise}.
\end{cases}
\end{equation*}
In particular, the composition factors $F(1-k)^\ga$ and 
$T(n+k-1)^\ga \simeq I(\triv,1-k)^\ga/F(1-k)^\ga$ of $I(\triv, 1-k)^\ga$ 
can be realized as
\begin{equation*}
F(1-k)^\ga = \Ker(\D_k)^\alpha
\quad
\text{and}
\quad
T(n+k-1)^\ga \simeq \Im(\D_k)^\alpha.
\end{equation*}

\end{thm}

\begin{proof}
Since the second assertion follows from the first and Theorem \ref{thm:vDM},
it suffices to consider the first statement.
As $\D_k \not\equiv 0$ for $k \in 1+\Z_{\geq 0}$, 
by Theorem \ref{thm:vDM},
there are only two possibilities on $\Ker(\D_k)^\alpha$, namely,
$\Ker(\D_k)^\alpha= \{0\}$ or  $F(1-k)^\ga$.
It then suffices to show that $\Ker(\D_k)^\alpha \neq \{0\}$ 
as far as $I(\triv,1-k)^\ga$ 
is reducible. 
In what follows, we understand that
$I(\triv,1-k)^\ga$ is a subspace of $C^\infty(\R^{n-1})$
via the isomorphism \eqref{eqn:n-} (see \eqref{eqn:21}).
In particular, the Lie algebra $\fg$ acts on $I(\triv,1-k)^\ga$ via the representation 
$\dpi_{1-k} \equiv \dpi_{(\alpha\otimes\triv,\chi^{1-k})}$ (see \eqref{eqn:dpi2}).

Suppose that $I(\triv, 1-k)^\ga$ is reducible. 
Since $F(1-k)^\ga$ is irreducible and 
finite-dimensional, there exists a lowest weight vector 
$f_0(x_1,\ldots, x_{n-1})\in F(1-k)^\ga$.
We shall show that $\D_k f_0=0$. 

As $f_0$ being a lowest weight vector of 
$F(1-k)^\ga \subset I(\triv, 1-k)^\ga$, 
we have
\begin{equation}\label{eqn:lw}
\dpi_{1-k} (N_j^-)f_0(x_1, \ldots, x_{n-1})=0 
\quad
\text{for all $j \in \{1,\ldots, n-1\}$}.
\end{equation}
A direct computation shows that
$\dpi_{1-k} (N_j^-) = - \frac{\partial}{\partial x_j}$.
Thus, \eqref{eqn:lw} is equivalent to
\begin{equation*}
\frac{\partial}{\partial x_j}f_0(x_1, \ldots, x_{n-1}) = 0
\quad
\text{for all $j \in \{1,\ldots, n-1\}$},
\end{equation*}
which shows that $f_0$ is a constant function. 
Therefore, we have
\begin{equation*}
\D_k f_0(x_1, \ldots, x_{n-1}) 
= \sum_{\mathbf{k} \in \Xi_k} 
\frac{\partial^k}{\partial x^{\mathbf{k}}}f_0(x_1, \ldots, x_{n-1}) 
\otimes 
\widetilde{y}_{\mathbf{k}}
=0.
\qedhere
\end{equation*}
\end{proof}

Corollary \ref{cor:IDO4} below is an immediate consequence 
of Theorems \ref{thm:vDM} and \ref{thm:IDO3}.

\begin{cor}\label{cor:IDO4}
For $(\alpha, k) \in \{\pm \} \times (1+\Z_{\geq 0})$,
the $K$-type formulas of $\Ker(\D_k)^\alpha_K$ and 
$\Im(\D_k)^\alpha_K$ are given as follows.
\begin{enumerate}
\item[\emph{(1)}] $\alpha = +\colon$
\vskip 0.1in
\begin{enumerate}
\item[\emph{(1-a)}] $k \in 1+2\Z_{\geq 0}\colon$
\begin{alignat*}{2}
\Ker(\D_k)^+_K &= F(1-k)^+_K && \simeq
\bigoplus_{\ell=0}^{(k-1)/2} \Cal{H}^{2\ell}(\R^n),\\[5pt]
\Im(\D_k)^+_K&\simeq T(n+k-1)^+_K &&\simeq \bigoplus_{\ell \geq \frac{k+1}{2}}
\Cal{H}^{2\ell}(\R^n).
\end{alignat*}
\item[\emph{(1-b)}] $k \in 2(1+\Z_{\geq 0})\colon$
\begin{alignat*}{2}
\Ker(\D_k)^+_K &=\{0\},&&\\[5pt]
\Im(\D_k)^+_K&\simeq I(\triv,1-k)^+_K &&\simeq \bigoplus_{\ell \in \Z_{\geq 0}}
\Cal{H}^{2\ell}(\R^n).
\end{alignat*}

\end{enumerate}

\item[\emph{(2)}] $\alpha = -\colon$
\vskip 0.1in
\begin{enumerate}
\item[\emph{(2-a)}] $k \in 1+2\Z_{\geq 0}\colon$
\begin{alignat*}{2}
\Ker(\D_k)^-_K &=\{0\},&&\\[5pt]
\Im(\D_k)^-_K&\simeq I(\triv,1-k)^-_K &&\simeq \bigoplus_{\ell \in \Z_{\geq 0}}
\Cal{H}^{2\ell+1}(\R^n).
\end{alignat*}
\item[\emph{(2-b)}] $k \in 2(1+\Z_{\geq 0})\colon$
\begin{alignat*}{2}
\Ker(\D_k)^-_K &= F(1-k)^-_K && \simeq
\bigoplus_{\ell=0}^{(k-2)/2} \Cal{H}^{2\ell+1}(\R^n),\\[5pt]
\Im(\D_k)^-_K&\simeq T(n+k-1)^-_K &&\simeq \bigoplus_{\ell \geq \frac{k}{2}}
\Cal{H}^{2\ell+1}(\R^n).
\end{alignat*}
\end{enumerate}

\end{enumerate}
\end{cor}

\section{Appendix: the standardness of the homomorphism $\varphi_k$}
\label{sec:Std}

The aim of this appendix is to show that the homomorphisms $\varphi_k$
in \eqref{eqn:Hom} are all standard maps. We achieve this in Theorem \ref{thm:std}.

\subsection{Standard map}\label{sec:Std1}

We start by introducing the definition of the standard map.
Let $\fg$ be a complex simple Lie algebra. Fix a Cartan subalgebra $\fh$ and 
write $\gD\equiv \gD(\fg, \fh)$ for the set of roots of $\fg$ with respect to $\fh$.
Choose a positive system $\gD^+$ and denote by $\Pi$ the set of simple roots of 
$\gD$. Let $\fb$ denote the Borel subalgebra of $\fg$ associated with $\gD^+$, namely,
$\fb = \fh \oplus \bigoplus_{\ga \in \gD^+} \fg_\ga$, where $\fg_\ga$ is the root space
for $\ga \in \gD^+$.

Let $\IP{\cdot}{\cdot}$ denote the inner product on $\fh^*$ induced from 
a non-degenerate symmetric bilinear form 
of $\fg$. For $\ga \in \gD$, we write $\ga^\vee = 2\ga/\IP{\ga}{\ga}$. Also, write 
$s_\ga$ for the reflection with respect to $\ga \in \gD$. As usual, we let 
$\rho=(1/2)\sum_{\ga \in \gD^+}\ga$ be half the sum of the positive roots.

Let $\fp \supset \fb$ be a standard parabolic subalgebra of $\fg$. Write 
$\fp =\fl \oplus \fn_+$ for the Levi decomposition of $\fp$. We let 
$\Pi(\fl) = \{\ga \in \Pi : \fg_\ga \subset \fl\}$.

Now we put
\begin{equation*}
\mathbf{P}^+_{\fl}:=\{\mu \in \fh^* : \IP{\mu}{\ga^\vee} \in 1+\Z_{\geq 0}
\;\;
\text{for all $\ga \in \Pi(\fl)$}\}.
\end{equation*}
For $\mu \in \mathbf{P}^+_\fl$, let $E(\mu-\rho)$ be the finite-dimensional 
simple $\Cal{U}(\fl)$-module with highest weight $\mu-\rho$. 
By letting $\fn_+$ act trivially, we regard $E(\mu-\rho)$ as a $\Cal{U}(\fp)$-module.
Then the induced module
\begin{equation*}
N_\fp(\mu):=\Cal{U}(\fg)\otimes_{\Cal{U}(\fp)} E(\mu-\rho)
\end{equation*}
is the generalized Verma module with highest weight $\mu-\rho$.
If $\fp = \fb$, then $N(\mu) \equiv N_{\fb}(\mu)$ is the (ordinary) Verma module with highest 
weight $\mu-\rho$.  

\vskip 0.1in


Let $\mu, \eta \in \mathbf{P}^+_\fl$.
It follows from 
a theorem by BGG-Verma 
(cf.\ \cite[Thm.\ 7.6.23]{Dix96} and \cite[Thm.\ 5.1]{Hum08})
that
if $\Hom_{\fg}(N_\fp(\mu), N_\fp(\eta)) \neq \{0\}$, then 
$\Hom_{\fg}(N(\mu), N(\eta)) \neq \{0\}$.

Conversely,
suppose that there exists a non-zero
$\fg$-homomorphism $\varphi\colon N(\mu)\to N(\eta)$.
Let $\pr_\mu\colon N(\mu) \to N_\fp(\mu)$ denote 
the canonical projection map. Then we have 
$\varphi(\Ker(\pr_\mu)) \subset \Ker(\pr_\eta)$
(\cite[Prop.\ 3.1]{Lepowsky77}). Thus, the map $\varphi$ induces a 
$\fg$-homomorphism $\varphi_{\std} \colon N_\fp(\mu)\to N_\fp(\eta)$ such that
the following diagram commutes.
\begin{equation*}
\xymatrix@R-=0.8pc@C-=0.5cm{
N(\mu) \ar[rr]^\varphi \ar[dd]_{\pr_\mu} && N(\eta \ar[dd]^{\pr_\eta})\\
&\circlearrowleft&\\
N_\fp(\mu)\ar[rr]^{\varphi_{\std}}& & N_\fp(\eta)
}
\end{equation*}

The map $\varphi_{\std}$ is called the \emph{standard map} from $N_\fp(\mu)$ to 
$N_\fp(\eta)$ (\cite{Lepowsky77}). As $\dim \Hom_{\fg}(N(\mu), N(\eta)) \leq 1$, 
the standard map $\varphi_{\std}$ is unique up to scalar. 
It is known that the standard map $\varphi_{\std}$ could be zero, and even if 
$\varphi_{\std}=0$, there could be another non-zero map from $N_\fp(\mu)$ to 
$N_\fp(\eta)$. Any homomorphism that is not standard is called 
a \emph{non-standard map}.

It is known when the standard map $\varphi_\std$ is zero.
To state the criterion, we first introduce the notion of a link between two weights.

\begin{defn}[{Bernstein--Gelfand--Gelfand}]\label{def:link}
Let $\mu, \eta \in \fh^*$ and $\beta_1,\ldots, \beta_t \in \gD^+$. Set
$\eta_0 := \eta$ and $\eta_i := s_{\beta_i}\cdots s_{\beta_1}\eta$ for $1\leq i \leq t$.
We say that the sequence $(\beta_1,\ldots, \beta_t)$ links $\eta$ to $\mu$ if 
it satisfies the following two conditions.
\begin{enumerate}
\item $\eta_t = \mu$.
\vskip 0.05in
\item $\IP{\eta_{i-1}}{\beta^\vee_{i}} \in \Z_{\geq 0}$ 
for all $i \in \{1,\ldots, t\}$.
\end{enumerate} 
\end{defn}

The criterion on the vanishing of the standard map $\varphi_{\std}$ is first
studied by Lepowsky (\cite{Lepowsky77}) and then Boe refined Lepowsky's 
criterion (\cite{Boe85}). 
Theorem \ref{thm:Boe} below is a version of Boe's criterion
\cite[Thm.\ 3.3]{Boe85}.

\begin{thm}\label{thm:Boe}
Let $\mu, \eta \in \fh^*$ and suppose that $\Hom_{\fg}(N(\mu), N(\eta)) \neq \{0\}$.
Then the following two conditions on $(\mu,\eta)$ are equivalent.
\begin{enumerate}
\item[\emph{(i)}]
The standard map $\varphi_{\std}\colon N_\fp(\mu) \to N_\fp(\eta)$ is non-zero.
\item[\emph{(ii)}]
For all sequences 
$(\beta_1, \ldots, \beta_t)$ linking $\eta$ to $\mu$,
we have $\eta_1 \in \mathbf{P}^+_\fl$.
\end{enumerate}
\end{thm}

\subsection{The standardness of the homomorphism $\varphi_k$}\label{sec:Std3}

Now we specialize the situation to the one considered in Section \ref{sec:SLn},
that is, $\fg = \f{sl}(n,\C)$ and $\fp$ is the maximal parabolic subalgebra 
corresponding to the partition $n=1 + (n-1)$.
Observe that if $\fg = \f{sl}(2,\C)$, then $N_\fp(\mu)=N_\fb(\mu)$. Thus, 
we assume that $n\geq 3$.

As usual, a Cartan subalgebra $\fh \subset \fg$ is taken to be
\begin{equation*}
\fh = \{\diag(a_1, \ldots, a_n): \sum_{i=1}^n a_i =0\}.
\end{equation*}
We take a set of positive roots $\gD^+$ and simple roots $\Pi$ as
\begin{equation*}
\gD^+=\{ \eps_i - \eps_j : 1 \leq i < j \leq n\}
\end{equation*}
and
\begin{equation*}
\Pi = \{\eps_i - \eps_{i+1} : 1 \leq i \leq n-1\}.
\end{equation*}
We realize $\fh^*$ as a subspace of $\R^{n}$ and write elements in $\fh^*$ in coordinates. For instance, we write $\eps_1-\eps_2 =(1,-1, 0,\ldots, 0)$.

Recall from Theorem \ref{thm:Hom2} that we have
\begin{equation}\label{eqn:Hom3}
\Hom_{\fg}(M_\fp(\sym^k_{n-1},-(1+\tfrac{k}{n-1})), M_\fp(\triv, k-1))
=\C \varphi_k
\end{equation}
for $k \in \Z_{\geq 0}$, where $\varphi_k$ is given in  \eqref{eqn:Hom}.
Let $d\chi \in \fh^*$ be the differential of the character $\chi$ of $A$ defined in 
\eqref{eqn:chi}. Then
\begin{equation*}
d\chi = \frac{1}{n}(n-1, -1, \ldots, -1).
\end{equation*}
Define $\omega_1 \in \fh^*$ by
\begin{equation*}
\omega_1=\frac{1}{n-1}(0, n-2, -1, \ldots, -1).
\end{equation*} 
For $\rho = (1/2)\sum_{\ga \in \gD^+}\ga \in \fh^*$, we have
\begin{equation*}
\rho = \frac{1}{2}(n-1, n-3, \ldots, -(n-3), -(n-1)).
\end{equation*}
For simplicity we let
\begin{align*}
\mu^k&= k\omega_1 -(1+\tfrac{k}{n-1})d\chi+\rho,\\[5pt]
\eta^k &= (k-1)d\chi + \rho.
\end{align*}
We have
\begin{align*}
N_\fp(\mu^k)
&=M_\fp(\sym^k_{n-1},-(1+\tfrac{k}{n-1})), \\[5pt]
N_\fp(\eta^k)
&= M_\fp(\triv, k-1).
\end{align*}
Equation \eqref{eqn:Hom3} is then given as
\begin{equation*}
\Hom_\fg
\left(
N_\fp(\mu^k),
N_\fp(\eta^k)
\right)
=\C\varphi_k.
\end{equation*}

\begin{thm}\label{thm:std}
Let $n\geq 3$. 
Then the homomorphism $\varphi_k$ is standard for all $k \in \Z_{\geq 0}$.
\end{thm}

\begin{proof}
Since $\dim \Hom_\fg(N_\fp(\mu^k), N_\fp(\eta^k)) =1$, it suffices to show 
the standard map $\varphi_{\std}^k \colon N_\fp(\mu^k)\to N_\fp(\eta^k)$ is non-zero.
We remark that 
$\Hom_\fg\left(N(\mu^k),N(\eta^k)\right)\neq \{0\}$
as $\Hom_\fg\left(N_\fp(\mu^k),N_\fp(\eta^k)\right)\neq \{0\}$.

First, suppose that $k=0$. In this case we have 
$\mu_0 = \eta_0= -d\chi+\rho$. 
Thus, the standard map $\varphi_\std^0$ is $\varphi_\std^0 = \id$; in particular,
$\varphi_\std^0$ is non-zero.

Next, suppose that $k \in 1+\Z_{\geq 0}$. In this case we have
\begin{align*}
\mu^k &= 
\frac{1}{2n}(2(1-k)+n(n-3), (n-1)(n-2+2k), 2(1-k)+n(n-5), \ldots, 2(1-k)-n(n-1)),\\[5pt]
\eta^k &=\frac{1}{2n}((n-1)(n-2+2k), 2(1-k)+n(n-3), 2(1-k)+n(n-5), \ldots, 2(1-k)-n(n-1)) .
\end{align*}
Then $\eps_1-\eps_2$ is the only element that links $\eta^k$ to $\mu^k$.
Since $\eta_1^k = s_{\eps_1-\eps_2}\eta^k=\mu^k \in \mathbf{P}^+_{\fl}$, 
Theorem \ref{thm:Boe} shows that the standard map $\varphi_{\std}^k$ is non-zero.
\end{proof}

\textbf{Acknowledgements.}
The authors are grateful to Dr.\ Ryosuke Nakahama and Dr.\ Masatoshi Kitagawa,
and Prof.\ Toshiyuki Kobayashi for fruitful communication on this paper.
They would also like to show their gratitude to the anonymous referees
to review the article carefully.
The first author was partially supported by JSPS
Grant-in-Aid for Scientific Research(C) (JP22K03362).


\bibliographystyle{amsplain}
\bibliography{KuOrsted}


\end{document}